\documentclass{article}
\usepackage{amsmath,amsfonts,amsthm,amscd,amssymb}
\usepackage{pstricks}
\usepackage{caption}
\usepackage{graphicx}
\usepackage{tabulary}
\usepackage{tikz-cd}
\newtheorem{thm}{Theorem}[section]
\newtheorem{pro}[thm]{Proposition}
\newtheorem{cor}[thm]{Corollary}
\newtheorem{lem}[thm]{Lemma}

\newtheorem{defn}[thm]{Definition}

\begin{document}

\author{Mark Herman, Jonathan Pakianathan and Erg\"un Yal\c c\i n}
\title{On a canonical construction of tesselated surfaces via finite group theory, Part I.
}
\maketitle

\begin{abstract}
This paper is the first part in a 2 part study of an elementary functorial construction from the category of finite non-abelian groups to a category of singular compact, oriented $2$-manifolds. After a desingularization process this construction results in a collection of compact, connected, oriented tesselated smooth surfaces equipped with a closed-cell structure which is face and edge transitive and which has at most 2 orbits of vertices. These tesselated  surfaces can also be viewed as abstract $3$-polytopes (or as graph embeddings in the corresponding surface) which are either equivar or dual to abstract quasiregular polytopes. This construction  generally results in a large collection of tesselated surfaces per group, for example when the construction is applied to $\Sigma_6$ it yields $4477$ tesselated surfaces of 27 distinct genus and even more varieties of tesselation cell structure.

We study the distribution of these surfaces in various groups and some interesting resulting tesselations. In a second paper, we show that extensions of groups result in branched coverings between the component  surfaces in their decompositions. We also exploit functoriality to obtain interesting faithful, orientation preserving actions of subquotients of these groups and their automorphism groups on these surfaces and in the corresponding mapping class groups.


\end{abstract}

\tableofcontents

\section{Introduction}

There is a long tradition in the area of algebraic combinatorics of associating geometries or simplicial complexes to algebraic objects such as groups as a means of studying them. In group cohomology, various combinatorial complexes associated to groups such as the Quillen complex and the Brown complex have proven very useful tools. (See \cite{Brown}, \cite{Be})

In this paper we explore such a construction which takes a finite nonabelian group $G$ and constructs a $2$-dimensional simplicial complex $X(G)$ in a canonical manner from it. This complex turns out to have a nice geometric structure as a union of finitely many 2-dimensional pseudomanifolds which pairwise intersect in finitely many points. The pseudomanifolds are compact, connected and oriented and hence surfaces with at most singularities due to self-point-intersections, for example a sphere with 3 points identified or a torus with three meridian circles squashed to points. Here the construction $X(G)$ is canonical in the sense
it is part of a covariant functor from the category of nonabelian finite groups to the category of compact, connected, oriented, 2-dimensional triangulated
manifolds with ``singularities" of a type described precisely in Section~\ref{sec: complex} in detail.

These singularities can be resolved in a functorial way to yield an associated complex $Y(G)$ which is a disjoint union of finitely many compact, connected, oriented, triangulated 2-manifolds (we'll call these Riemann surfaces in this paper for brevity at the expense of slight abuse of notation as we will not be using complex structure at all.) The manifolds in $Y(G)$ hence form a natural set of invariants for the group. More specifically the number of times the surface of genus $g$ occurs in $Y(G)$ is a natural invariant of the group which we study. Furthermore the original
complex $X(G)$ can be shown to be homotopy equivalent to a bouquet of these manifolds with
some circles due to the connections between them.

Additionally, the Riemann surfaces arising via this functorial construction are equipped with interesting closed-cell structures naturally associated to
their triangulations and indeed we show that a slight functorial modification of the triangulations of the components yield closed-cell structures where all faces consist of a given type of polygonal face. We show these satisfy the conditions of what is called an abstract 3-polytope in the combinatorics literature. Any of these component polytopes is shown to have a face and edge transitive automorphism group and at most 2 orbits of vertices (and hence two types of vertex valency). We identify group theoretic conditions for a component polytope to be equivar (have a single valency) and identify the Schl\"afli symbol in this case. We also identify conditions for when the component forms a regular abstract 3-polytope such as a Platonic solid. In this case the component can also be considered as a regular map, i.e., a regular embedding of a graph in the underlying surface. For example this is the case when $G$ is the extraspecial $p$-group of order $p^3$ where $p$ is an odd prime which results in a regular tesselation of the surface of genus $g=\frac{p(p-3)}{2}+1$ as the union of $p$ $2p$-gons.

Constructions that yield such tesselations on Riemann surfaces are not new. There are classical constructions using triangle or Fuchsian groups,
hyperbolic geometric tesselations of the Poincare disk, Cayley graphs or fundamental group arguments to construct similar objects (see \cite{Cox}, \cite{Cox2} for example). Often a set of generators or specific presentation is chosen - our construction seems like an elementary variant, constructed in an algorithmic and functorial way from any nonabelian group.
The initial construction uses the group itself in a very direct straightforward manner with no reference to specific presentation or external geometry - it clearly displays these tesselations as intrinsic to the group itself.

On the other hand, since $X(G)$ and $Y(G)$ are determined functorially from $G$, one
can show that any automorphism of $G$ determines an orientation preserving, simplicial automorphism of $X(G)$. These simplicial automorphisms can permute the surface ``components" of $Y(G)$ but must preserve the genus and indeed the tessellation type. Thus, one can use this construction to find faithful actions of subquotients of $Aut(G)$
on various  surfaces and hence obtain embeddings of these subquotients into the group of orientation preserving diffeomorphisms of $X_g$, the surface of genus $g$.
In the second paper \cite{HP}, we explore some examples of orientation preserving group actions afforded by this construction and some corresponding embeddings into mapping class groups. We also discuss some examples pertaining to the strong symmetric genus (see \cite{TT}, \cite{MZ1}) of the group, i.e. the smallest genus orientable closed surface for which the group acts on faithfully via orientation preserving diffeomorphisms. See \cite{HP} for details.

We attempt to quantify the number and type of surface components that occur in the decomposition $Y(G)$ for a given group $G$. For example when $G$ is the alternating group on $7$ letters, the construction $Y(G)$ discussed in this paper results in $16813$ surface components of $58$ distinct genus, with even more
distinct cell-structures. Tables summarizing
the various polytopes that occur in $Y(G)$ for various groups $G$ can be found in the appendix
of \cite{HP}.

Finally in \cite{HP}, we show that extensions of groups yield branched coverings between their constituent surfaces in their decompositions.

\section{The complex $X(G)$}
\label{sec: complex}

Let $G$ be a finite nonabelian group and let $Z(G)$ be its center. We will now describe the
construction of $X(G)$ in detail - the reader is encouraged to draw pictures to help
with the definitions in this section as at first they seem elaborate though with some study, the nice structure of the complex will emerge.
We will use some basic notation from the world of simplicial complexes - for a good reference on the terminology, see \cite{Mu}.

Before we construct the simplicial complex $X(G)$, we construct an associated subgraph $G_1=(V_1,E_1)$.
The vertices of the graph $G_1$ are given by the noncentral elements of $G$, i.e., $V_1=G-Z(G)$ and are labelled as type 1 vertices, thus a typical vertex is labelled $(x,1)$ where $x \in G-Z(G)$. We declare $[(x,1), (y,1)]$ to be an edge in the simplicial graph $G_1$ if and only if $x$ and $y$ do not commute in $G$, i.e., $xy \neq yx$.

This graph is the 1-skeleton of the simply connected complex $BNC(G)$ considered in \cite{PY} and hence is path connected but we provide a simple proof of this here:

\begin{lem}
$G_1=(V_1,E_1)$ is a path connected graph.
\end{lem}
\begin{proof}
Let $(x,1), (y,1)$ be two vertices in $G_1$. Then $x, y$ are not central in $G$ so their centralizer subgroups $C(x)$ and $C(y)$ are proper subgroups
of $G$ and hence have at most half the elements of $G$ in them. Thus $|C(x) \cup C(y)| < |G|$ as $C(x),C(y)$ each have at most half the elements of $G$ in them
and they have the identity element in common. Thus there must exist $z \notin C(x) \cup C(y)$ which means there is a path of length two from $(x,1)$ to $(y,1)$
in $G_1$ through $(z,1)$ as $z$ does not commute with either $x$ or $y$. Since $(x,1), (y,1)$ were arbitrary vertices of the graph $G_1$, we conclude
that $G_1$ is path connected.
\end{proof}

We now extend the graph $G_1$ to get the simplicial complex $X(G)$. First we include a second set of vertices $V_2$ which also consists of the noncentral
elements of $G$ but with label 2, written as $(x,2)$, for $x \in G-Z(G)$. Thus the vertex set of $X(G)$ consists of the disjoint union of $V_1$ and $V_2$
i.e. two copies of the non central elements of $G$. To describe a simplicial complex, it is enough to describe its maximal faces.

The maximal faces of $X(G)$ are of the form $[(x,1), (y,1), (xy,2)]$ where $x$ and $y$ do not commute - we orient these $2$-simplices in the order indicated also.
Note that the three vertices $x, y, xy$ of this face pairwise do not commute in $G$.

Notice every edge in the graph $G_1$ lies in two distinct faces of $X(G)$ i.e., \\  $[(x,1),(y,1), (xy,2)]$ and $[(y,1), (x,1), (yx,2)]$. Also note that the orientations on these two faces cancel along
this edge.
Furthermore if $\alpha$ is a noncentral element, there must be an element $x$ which does not commute with it by definition. Then $y=x^{-1}\alpha$ does not
commute with $x$ and $[(x,1), (y,1), (\alpha,2)]$ is a face of $X(G)$. Thus every type 2 vertex does lie in a face also. Notice from this argument
that for any $\alpha, x$ which do not commute, the edge $[(x,1), (\alpha,2)]$ is part of the complex. Furthermore such an edge lies again in exactly
two faces $[(x,1), (x^{-1}\alpha,1), (\alpha,2)]$ and $[(\alpha x^{-1},1), (x,1), (\alpha,2)]$ whose
orientations cancel along that edge.

\begin{figure}[htp]
\centering
\caption{The face $[(x,1), (y,1), (xy,2)]$ and two adjacent faces}
\includegraphics[width=.4 \textwidth]{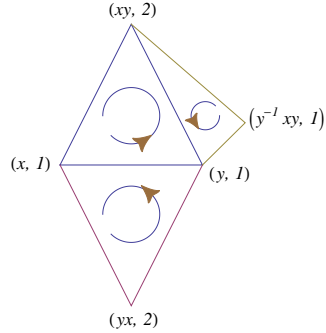}
\label{fig: 2face}
\end{figure}

To summarize, we have defined a $2$-dimensional simplicial complex $X(G)$ whose vertex set $V$ is the disjoint union of $V_1$ and $V_2$ where
$V_1=\{ (x,1) | x \in G-Z(G) \}$ and $V_2=\{ (x,2) | x \in G-Z(G) \}$ and whose edges are of the form $[(x,1), (y,1)]$ and $[(x,1), (y,2)]$ where $x$ and $y$ do not commute.
Finally the (oriented) faces in the complex are of the form $[(x,1),(y,1),(xy,2)]$ and every edge lies in exactly two faces whose orientations cancel along that edge. Note in particular every face
of $X(G)$ shares an edge with the graph $G_1$ considered earlier.

Finally note that $X(G)$ is path-connected as any point in $X(G)$ lies in a face and any point in a face can be connected by a straight line path to a point on
the graph $G_1$. As we have seen that the graph $G_1$ is path-connected, it follows that any two points in $X(G)$ are connected by a path.

We summarize the properties of $X(G)$ here that will be relevant to the rest of the analysis of this complex in this section. We postpone the proof of the functoriality of this construction till later.

\begin{pro}
\label{basic properties}
Let $G$ be a finite nonabelian group and $Z(G)$ be its center. The complex $X(G)$ is a
compact, 2-dimensional simplicial complex such that: \\
(1) Every edge lies in exactly two faces. \\
(2) The faces can be oriented so that along any edge, the two orientations of the adjacent
faces cancel. \\
(3) It is path-connected, thus in particular every vertex lies in an edge.
\end{pro}

We will now show that any simplicial complex that satisfies the criteria of
Proposition~\ref{basic properties} like $X(G)$ does is a union of finitely many compact, connected, oriented,
pseudomanifolds who pairwise intersect in a finite collection of vertices.

\subsection{Pseudomanifolds}

Recall (see \cite{Mu}) that an $n$-dimensional pseudomanifold is a $n$-dimensional simplicial complex $X$ which
satisfies: \\
(1) $X$ is the union of its $n$-simplices. \\
(2) Every $n-1$ simplex in $X$ lies in exactly two $n$-simplices. \\
(3) For any two $n$-simplices $\sigma_0, \sigma_k$ in $X$ there is a sequence of $n$-simplices
$\sigma_1, \dots, \sigma_{k-1}$ such that $\sigma_i$ and $\sigma_{i+1}$ share a common
$(n-1)$-face for $0 \leq i \leq k-1$. \\
Furthermore a pseudomanifold is orientable if one can orient all the $n$-simplices in such a way
that for any $n-1$ simplex $\tau$ in $X$, the two orientations from the two adjacent $n$-simplices cancel along $\tau$. \\
Note that conditions (1) and (3) imply that all pseudomanifolds are path-connected. \\

It is well known that every path connected, triangulated topological $n$-manifold is a $n$-dimensional pseudomanifold. However pseudomanifolds also allow certain types of singularites which we will talk about below.

\begin{pro} Let $X$ be a simplicial complex satisfying the conditions of Proposition~\ref{basic properties} such as $X(G)$. Define an equivalence relation on the $2$-faces of $X$ by
saying $\sigma_0$ and $\sigma_k$ are equivalent if there is a sequence of $2$-faces
$\sigma_1, \dots, \sigma_{k-1}$ such that $\sigma_i$ and $\sigma_{i+1}$ share a common
edge for $0 \leq i \leq k-1$.

The union of faces in an equivalence class then gives an oriented compact $2$-dimensional pseudomanifold. Furthermore $X$ has finitely many equivalence classes and so is a union of
finitely many oriented, compact, pseudomanifolds any two of which can meet only in a finite
set of vertices.
\end{pro}
\begin{proof}
As every vertex of $X$ lies in an edge and every edge lies in a face ($2$-simplex), it is clear that $X$ is the union of its $2$-simplices. The relation described on these faces is easily checked to be an equivalence relation and the equivalence classes partition the $2$-simplices of $X$. As every edge of $X$ lies in exactly two faces and faces can be oriented so that orientations cancel along the edges in common, it is clear that the union of the $2$-simplices in one of the equivalence classes forms a $2$-dimensional oriented pseudomanifold. It consists of finitely many simplices and is compact as $X$ is. Thus $X$ is a union of finitely many compact, oriented, $2$-dimensional pseudomanifolds as claimed. If $Y$ and $Z$ are two of these, the intersection $Y \cap Z$ is a
subcomplex. If this complex contained any edge, it would bound two faces, one of which has to be in $Y$ and one which has to be in $Z$, and these faces would be equivalent contradicting
the construction of the pseudomanifolds $Y$ and $Z$ as unions from distinct equivalence classes of faces. Thus the intersection contains no edges and hence no faces and so consists of at most a finite collection of vertices as claimed.
\end{proof}

Note from the proof above, it is clear that the pseudomanifolds whose union is $X(G)$ can be algorithmically determined and hence are unique. We will call them the pseudomanifold components of $X(G)$.

The following three examples are examples of simplicial complexes satisfying the hypothesis of Proposition~\ref{basic properties} and give a good indication of the type of spaces that can be $X(G)$:\\
(1) A compact, connected, oriented, triangulated $2$-manifold. Here there is a single pseudomanifold component which happens to be a
manifold. We will call these examples, with slight abuse of notation, ``Riemann surfaces" - in \cite{HP}, we show that our constructions inherit unique complex 
structures compatible with functoriality so this abuse is justifiable.\\
(2) Any compact, oriented $2$-dimensional pseudomanifold. For example a sphere with 3 points identified to a single point.
This can be triangulated so that it is a compact, oriented pseudomanifold and is a typical example of such.
Every compact, oriented 2-dimensional pseudomanifold can be obtained from a triangulated
Riemann surface by a finite number of vertex identifications. We will see this soon as a byproduct of other analysis. Thus basically in dimension $2$ an oriented compact pseudomanifold is nearly  an oriented compact connected manifold aside from some point self-intersections. \\
(3) Take a torus and squash three meridian circles to three distinct points. This space can be triangulated so that it satisfies the hypothesis of Proposition~\ref{basic properties}. It has
three pseudomanifold components which are spheres and any two of these pseudomanifold components intersect in a single vertex but no point lies in all three pseudomanifold components. \\

Let $X$ be any simplicial complex satisfying the conditions of Proposition~\ref{basic properties}.
Notice that if $x$ is a point in the interior of a face (2-simplex) then it is clear it has an open neighborhood homeomorphic to the unit open ball of $\mathbb{R}^2$. If $x$ is a point on an
edge other than its two end points, then as every edge lies in exactly two faces, it is again clear that $x$ has an open neighborhood homeomorphic to the unit open ball of $\mathbb{R}^2$.

Thus the only points of $X$ that might not have open neighborhoods homeomorphic to open
subsets of $\mathbb{R}^2$ are the vertices. In particular $X-\{ \text{Vertex set of } X\}$ is a
$2$-dimensional oriented topological manifold whose connected components, we will see later,  are punctured
Riemann surfaces. The next proposition shows that the connected components of
$X - \{\text{ vertex set of } X \}$ are in natural bijective correspondence with the pseudomanifold components of $X$.

\begin{pro}
\label{pro: components}
Let $X$ be a simplicial complex satisfying the conditions of Proposition~\ref{basic properties} and let $Y= X - \{ \text{Vertex set of } X \}$ be given the subspace topology. Then $Y$ is a $2$-dimensional oriented topological manifold and the connected components of $Y$ are in natural bijective correspondence with the pseudomanifold components of $X$. More precisely $y_1, y_2 \in Y$ lie in the same component of $Y$ if and only if they lie in the same pseudomanifold component of $X$.
\end{pro}
\begin{proof}
It is clear from the preceeding paragraph that $Y$ is an oriented $2$-dimensional topological manifold. Fix $y_1, y_2$ in $Y$. As $X$ was the union of its pseudomanifold components $X_1, \dots X_m$
which pairwise intersected in at most finitely many vertex points, we see that $Y$ is the disjoint union of $X_1-V, \dots, X_m-V$ where $V$ is the vertex set of $X$. As each $X_i$ is closed in $X$,
each $X_i-V$ is closed in $Y=X-V$ and hence also open as there are a finite number of pseudomanifold components. Thus if $y_1, y_2$ lie in distinct pseudomanifold components of $X$ they must lie in distinct connected components of $Y$. Conversely, suppose that $y_1, y_2$ lie in the same pseudomanifold component $X_j$ of $X$. Then there must be a face $\sigma_0$
and a face $\sigma_k$ both in $X_j$ joined by a sequence of faces $\sigma_1,\dots, \sigma_{k-1}$ in $X_j$ such that $y_1 \in \sigma_0$ and $y_2 \in \sigma_k$ and such that
$\sigma_i, \sigma_{i+1}$ share a common edge for $0 \leq i \leq k-1$. As $y_1, y_2$ are not
vertices, it is clear we can construct a path along this sequence of faces which avoids vertices and joins $y_1, y_2$. Thus $y_1, y_2$ lie in the same path component and hence component of $Y$. This concludes the proof.
\end{proof}

\subsection{Closed stars of vertices}

Throughout this section, we will be concerned with the structure of the closed star of vertices
in a simplicial complex $X$ which satisfies the conditions of Proposition~\ref{basic properties}.
Throughout this section $X$ is such a complex.

\begin{defn}[$m$-stars]
Let $m \geq 3$. Let $V=\{0\} \cup \{ e^{2\pi ik/m} | 0 \leq k < m, k \text{  an integer } \}$ be the
vertex set containing zero and the $m$th roots of unity viewed within the complex plane. Consider the collection of edges $E$ obtained by joining $0$ to the $m$, $m$th roots of unity. This simplicial graph $G_m=(V,E)$ will be called the $m$-star graph. Note it is star-convex so as a space, it is
contractible to its middle point $0$.
\end{defn}

\begin{defn}[$m$-disks]
Let $m \geq 3$. Let $V$ be the vertex set containing $0$ and the $m$th roots of unity viewed within the complex plane. Let $D_m$ be the convex hull of $V$ (i.e, a $m$-gon) triangulated
using its $m$ boundary edges and the $m$ edges joining $0$ to the $m$th roots of unity as edge set.
$D_m$ has $m$ faces ($2$-simplices) consisting of the $m$ triangles cut out by the edges mentioned above. $D_m$ is homeomorphic to the standard closed unit disk in the complex plane but triangulated by this specific triangulation. We will call such a triangulated disk, an $m$-disk or a disk of type $m$ in this paper. $0$ is called the center of this $m$-disk. The boundary of a $m$-disk is a triangulated circle which we will call an $m$-circle. \end{defn}

\begin{defn}[$(m_1,m_2,\dots,m_k)$-disk bouquet]
Let $m_1, \dots, m_k$ be integers $\geq 3$. A simplicial complex $T$ is called a
$(m_1,m_2,\dots, m_k)$-disk bouquet if it is simplicially isomorphic to the simplicial complex obtained by taking $k$ disjoint disks of types $m_1, \dots, m_k$ respectively and identifying
their centers to a common center vertex. The individual disks in a disk bouquet are called the sheets of the bouquet.
\end{defn}

We have seen neighborhoods of nonvertex points of $X$ look like open disks in $\mathbb{R}^2$, we are now ready to describe neighborhoods of vertices in $X$, including their simplicial structure.

\begin{pro}[Closed stars of vertices]
\label{pro: closedstar}
Let $X$ be a simplicial complex satisfying the conditions of Proposition~\ref{basic properties}
and let $v$ be a vertex of $X$. Then there exist integers $k \geq 1, m_1, \dots, m_k \geq 3$
such that the closed star of $v, \bar{St}(v)$ is a $(m_1,m_2, \dots, m_k)$-disk bouquet.
The link of $v, Lk(v)$ is a disjoint union of $k$ circles, of simplicial types $m_i$, $1 \leq i \leq k$.

\end{pro}
\begin{proof}
The vertex $v$ is contained in a face $\sigma_1$. Let $E$ be the edge opposite $v$ in this
face. Note that by the structure of simplicial complexes, no other face can contain both $v$ and $E$ or it would have to be equal to $\sigma_1$. Now take an edge of $\sigma_1$ that contains $v$.
As every edge lies in exactly two faces, there is a unique face $\sigma_2$ that shares this
edge with $\sigma_1$ and is not $\sigma_1$. Note $v$ is contained in this face also. Repeating the argument we can find a sequence of faces, each containing $v$ of the form
$\sigma_1, \dots, \sigma_k$ where $\sigma_i$ and $\sigma_{i+1}$ share an edge for
$1 \leq i \leq k-1$. Eventually by finiteness and as every edge lies in exactly two faces, we must have $\sigma_k$ share an edge with
the first face $\sigma_1$ for some $k=m_1$. The union of faces $\sigma_1, \dots, \sigma_{m_1}$ hence forms a $m_1$-disk centered at $v$. If this exhausts all faces in $X$ we are done. If not pick another unused face containing $v$ and proceed to find a $m_2$ disk centered at $v$ using a similar process. Proceed in this way till one has found $k$ disks all centered at $v$ and there are no unused faces containing $v$. (This must occur eventually as there are finite number of faces in $X$).

Note that as no two faces which have $v$ as a vertex, can intersect in an edge opposite $v$, and all the edges adjacent to $v$ can only be in two faces which are part of the same disk, it is clear that the disks obtained are disjoint except for the common vertex $v$. Thus the union of all faces containing $v$, i.e. the closed star of $v$, forms a $(m_1,m_2, \dots,m_k)$-disk bouquet. $m_i \geq 3$
for all $i$ by general structure conditions of simplicial complexes.

The statement on links follows immediately from this so we are done.

\end{proof}

\begin{cor}
Let $X$ be a simplicial complex satisfying the conditions of Proposition~\ref{basic properties}
and let $Y=X-\{ \text{Vertex set of } X \}$. Then the components of $Y$ are punctured Riemann surfaces. We may fill in each distinct puncture of $Y$ with a distinct point to obtain a
disjoint union of finitely many Riemann surfaces which are in bijective correspondence with the pseudomanifold components of $X$.
\end{cor}
\begin{proof}
We already know that $Y$ is an oriented $2$-dimensional manifold with finitely many connected components in bijective correspondence with the pseudomanifold components of $X$.
Note from Proposition~\ref{pro: closedstar}, we see that there is open neighborhood of each puncture homeomorphic to a punctured disk. From this it is easy to see that if the punctures are filled in with distinct points, we will obtain components which are connected, oriented
$2$-manifolds with finite triangulations (and hence compact). Adding these distinct points do not change connected components so we result in a finite disjoint union of Riemann surfaces, which are in bijective correspondence with the pseudomanifold components of $X$.
\end{proof}

\begin{figure}[htp]
\centering
\caption{A representative simplicial complex $X(G)$.}
\includegraphics[width=.8 \textwidth]{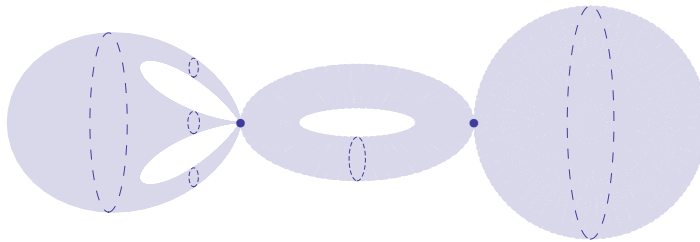}
\label{fig: disingularizedpre}
\end{figure}

\begin{figure}[htp]
\centering
\caption{The resulting desingularization $Y(G)$ of the complex $X(G)$ in figure \ref{fig: disingularizedpre}.}
\includegraphics[width=1 \textwidth]{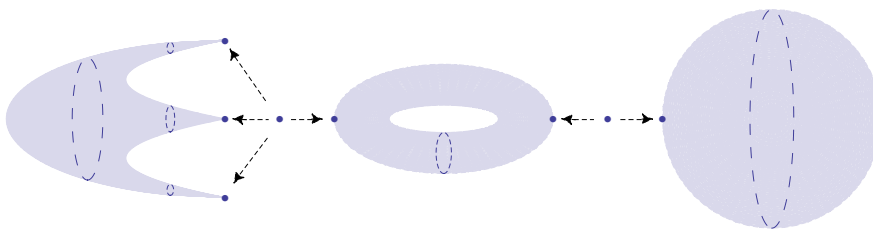}
\label{fig: disingularizedpost}
\end{figure}

\begin{defn}
Let $G$ be a finite nonabelian group. Then $X(G)$ satisfies the conditions of Proposition~\ref{basic properties} and so we may remove its vertices and fill in the resulting punctures with distinct points. The resulting space will be denoted by $Y(G)$ and is a disjoint union of Riemann surfaces which are in bijective correspondence with the pseudomanifold components of $X(G)$. We will see later that $Y(G)$ can be triangulated in a way that preserves the basic functoriality of the construction.

As a Riemann surface is determined up to homeomorphism by its genus $g$ which is a nonnegative integer, the quantity $m_g(G)$, the number or Riemann surfaces of genus $g$ that occur in the complex $Y(G)$ is a natural invariant of the group $G$ which we will study later.

$Y(G)$ will be called the desingularization of $X(G)$.
\end{defn}

\subsection{Main Structure Theorem}

In this subsection, we will prove the main structure theorem for the complex $X(G)$.

\begin{thm}
\label{thm: structurethm}
Let $G$ be a finite nonabelian group. Let $m_g(G)$ denote the number of surfaces of genus $g$
that occur in $Y(G)$, the desingularization of $X(G)$. Then $X(G)$ is homotopy equivalent
to a wedge product (bouquet) of Riemann surfaces and a finite number of circles where the surface of genus $g$ occurs
$m_g(G)$ times in the wedge product. In other words,
$$
X(G) \simeq (\bigvee_{g \geq 0} X_g^{[m_g]}) \bigvee (S^1 \vee \dots \vee S^1)
$$
where $X_g^{[m_g]}$ is the $m_g=m_g(G)$-fold bouquet of the Riemann surface $X_g$ of genus $g$ with itself.
\end{thm}
\begin{proof}
Let $Z(G)$ be the simplicial complex obtained as follows: First take out all the vertices from $X(G)$ resulting in a punctured oriented topological $2$-manifold. Label each puncture point by the vertex in $X(G)$ it resulted from (note if the closed star of $v$ in $X(G)$ is a bouquet of $k$ disks, there will be $k$ punctures with label $v$). Now for each vertex $v$ in $X(G)$ which lies
in $k$ sheets, i.e. whose closed star is a bouquet of $k$ disks, take a disjoint $k$-star graph and identify its ends with the $k$ points in the desingularization $Y(G)$ filling in the punctures labelled by $v$. The result is a path connected space consisting of the Riemann surfaces
in $Y(G)$ connected by a finite number of star graphs (any two of which are disjoint from each other), one for each vertex in $X(G)$. This is the space $Z(G)$.

Now note that if one collapses all the star graphs in $Z(G)$ to their central points, one obtains a homotopy equivalence $Z(G) \simeq X(G)$. Now $Z(G)$ consists of the finite set of Riemann surfaces $X_1, \dots, X_m$ of $Y(G)$ connected together with various star graphs.
Note that any edge in the star graph can be thought of as obtained by adjunction of the unit interval via a gluing map on its two boundary points. By basic facts on adjunction spaces, we can move the point of attachment of the interval along a continuous path without changing the homotopy type of the whole space.  Thus if one has a $k$-star graph in $Z(G)$ with $k \geq 3$, note that two of the edges form an interval whose end points lie on the Riemann surfaces. Thus using adjunction deformations, we can move the central point attachment of the other $k-2$ edges so that they attach to Riemann surfaces in $Y(G)$ on  both ends. Thus up to homotopy equivalence, $Z(G)$ and hence $X(G)$ is homotopy equivalent to $Y(G)$ with a finite number of edges attached where both end points lie in $Y(G)$ and whose interiors are disjoint. Now fix a basepoint $x_i$ in each Riemann surface $X_i$ of $Y(G)$. Using adjunction space deformations, we can up to homotopy equivalence assume all the edges have endpoints in the set $\{x_1, \dots, x_m\}$ as the Riemann surfaces $X_i$ are path connected.

At this stage as $X(G)$ is connected, the union of these edges forms a connected graph with
vertices $\{ x_1, \dots, x_m\}$. Collapsing a spanning tree of this graph (which is contractible) one obtains a final homotopy equivalence. Notice that under this collapse, the points $x_1, \dots, x_m$ become a common bouquet point to which the Riemann surfaces $X_1, \dots X_m$ from the desingularization $Y(G)$ are attached. Any edges of the graph not in the spanning tree of the graph become circles in this bouquet. The theorem is hence proven.

\end{proof}

\begin{cor}
\label{cor: Betti}
The integral homology groups of $X(G)$ are free abelian groups of finite rank. If $\beta_i$
denotes the $i$th Betti number, i.e., the rank of $H_i(X(G);\mathbb{Z})$ then \\
$\beta_0=1, \beta_2=\sum_{g \geq 0} m_g(G)$ and
$$\beta_1=\sum_{g \geq 0} 2gm_g(G) + L$$ where $m_g(G)$ denotes the number of times the surface of genus $g$ occurs in the desingularization $Y(G)$ of $X(G)$ and $L$ denotes the number of circles occurring in the homotopy decomposition of Theorem~\ref{thm: structurethm}.
\end{cor}
\begin{proof}
Follows immediately from Theorem~\ref{thm: structurethm} and the known homology of Riemann surfaces.
\end{proof}

\subsection{The Euler characteristic of $X(G)$}

In this section we find some useful formulas for the Euler characteristic of $X(G)$.
One of these follows immediately from the homology computation of the last section and the other from a simplex count of the simplicial complex which we go over now.

Vertices of $X(G)$: There are $|G|-|Z(G)|$ many vertices of type 1 and the same number of vertices of type 2 for a total of $V=2(|G|-|Z(G)|)$ vertices in $X(G)$.

Let $E_1$ denote the number of edges joining two type 1 vertices and let $E_2$ denote the number of edges joining a type 1 and type 2 vertex. Finally let $F$ denote the number of faces (2-simplices). As each face consists of 3 edges, (two in $E_2$ and one in $E_1$) and every edge lies in exactly two faces, we conclude $2F=2E_2$ and $F=2E_1$ and $E=E_1+E_2$ is the total number of edges. Thus the Euler characteristic of $X(G)$, $\chi(X(G))=V-E+F=V-(E_1+E_2)+E_2=V-E_1$.

Note that there is an edge in $E_1$ for every (unordered) set of two noncommuting elements of $G$. Thus
$E_1=\frac{1}{2} \sum_{x \in G} (|G|-|C(x)|)$ as for each $x \in G$, there are $|G|-|C(x)|$ elements $y$ which don't commute with $x$. Here $C(x)$ is the centralizer subgroup of $x$.

So $E_1=\frac{1}{2}(|G|^2-\sum_{x \in G}|C(x)|)$. The following lemma evaluates this sum:

\begin{lem}
\label{lem: centralizer sum}
For $G$ a finite group, $\sum_{x \in G} |C(x)| = |G|c$ where $c$ is the number of conjugacy classes of $G$.
\end{lem}
\begin{proof}
Ler $x_1,\dots,x_c$ denote a complete list of conjugacy class representatives of $G$.
Note that $|C(x)|=|C(x')|$ when $x,x'$ lie in the same conjugacy class. Thus
$$
\sum_{x \in G} |C(x)| = \sum_{i=1}^c |C(x_i)|c_i
$$
where $c_i$ is the size of the conjugacy class of $x_i$. Note that $c_i=\frac{|G|}{|C(x_i)|}$ and so
$$
\sum_{x \in G} |C(x)|=\sum_{i=1}^c |G|=|G|c
$$
as we set out to show.

\end{proof}

We record a side corollary of Lemma~\ref{lem: centralizer sum}:

\begin{cor}
Let $G$ denote a finite group and $P$ denote the probability that two elements picked independently and uniformly from $G$ commute. Then
$P=\frac{c}{|G|}$ where $c$ is the number of conjugacy classes of $G$.
\end{cor}
\begin{proof}
As any ordered pair of elements in $G \times G$ are equally likely to be selected we have
$$
P=\frac{\text{Number of }(x,y) \text{ where } x, y \text{ commute} }{|G|^2}
=\frac{sum_{x \in G} |C(x)|}{|G|^2}=\frac{c}{|G|}.
$$
where we used Lemma~\ref{lem: centralizer sum} in the last step.
\end{proof}

We now record two formulas regarding the Euler characteristic of $X(G)$.

\begin{thm}
Let $G$ be a finite nonabelian group. Then
$$
V-E+F=\chi(X(G))=\beta_0-\beta_1+\beta_2
$$
becomes
$$
2(|G|-|Z(G)|)-\frac{1}{2}(|G|^2-|G|c)=\chi(X(G))=1-L+\sum_{g \geq 0} (1-2g)m_g(G)
$$
where $c$ is the number of conjugacy classes of $G$ and $m_g(G)$ denotes the number of times
the surface of genus $g$ occurs in $Y(G)$ the desingularization of $X(G)$. $L$ denotes the number of circles that occur in the homotopy decomposition of $X(G)$ given by Theorem~\ref{thm: structurethm}.
\end{thm}
\begin{proof}
Follows as the Euler characteristic can either be computed by the alternating sum $V-E+F$ or the alternating sum of Betti numbers. The Betti sum is evaluated using corollary~\ref{cor: Betti} while $V-E+F=V-E_1$ is evaluated using the calculations made earlier in this subsection.
\end{proof}

\subsection{Group Theoretic Analysis of Closed Stars of $X(G)$}

\begin{defn}
Let $G$ be a finite group and $\alpha \in G$. The cyclic subgroup $C_{\alpha}=<\alpha>$
generated by $\alpha$ acts on $G$ by conjugation. The orbits are called $\alpha$-conjugacy classes and have size dividing the order of $\alpha$. Two elements of the group are said to be
$\alpha$-conjugate if they lie in the same $\alpha$-conjugacy class. An element $\{x\}$
forms an $\alpha$-conjugacy class of size $1$ if and only if $x$ commutes with $\alpha$
if and only if $x \in C(\alpha)$, the centralizer group of $\alpha$. We will denote
the conjugate of $x$ by $\alpha^{-1}$ i.e., $\alpha^{-1}x\alpha$ by $x^{\alpha}$.
\end{defn}

Let us study the closed star of a type 2 vertex $v=(\alpha,2)$ in $X(G)$. A face in $\bar{St}(v)$
is of the form $\sigma_1=[(x,1),(y,1),(\alpha,2)]$ where $x,y$ do not commute and $\alpha=xy$.
Note $\alpha$ does not commute with either $x$ or $y$.
Now let us consider the adjacent face $\sigma_2=[(y,1), (z,1), (\alpha,2)]$. Then $yz=\alpha$
and so $yz=xy$ i.e. $z=y^{-1}xy=y^{-1}x^{-1}xxy=\alpha^{-1}x\alpha=x^{\alpha}$.
Thus the adjacent face is $\sigma_2=[(y,1),(x^{\alpha},1), (\alpha,2)]$. Now taking the equation
$\alpha=xy$ corresponding to the first face $\sigma_1$ and conjugating it by $\alpha^{-1}$ we
see that $\alpha=x^{\alpha}y^{\alpha}$. Thus the face $\sigma_3$ other than $\sigma_1$ which is adjacent to $\sigma_2$
which contains the vertex $(\alpha,2)$ is $\sigma_3=[(x^{\alpha},1),(y^{\alpha},1),(\alpha,2)]$.
Note that $\sigma_3$ was obtained from $\sigma_1$ by conjugating its 3 vertices by
$\alpha^{-1}$.

We now have made an important observation. In a given sheet of the closed star of a type 2 vertex, given one triangle in the sheet, the triangle two away in the same sheet is obtained by conjugating all vertices by $\alpha^{-1}$. Thus the entire sheet is made by
$\alpha$-conjugacy applied to the two initial triangles $[(x,1),(y,1),(\alpha,2)]$ and
$[(y,1),(x^{\alpha},1),(\alpha,2)]$.

\begin{figure}[htp]
\centering
\caption{A representative sheet in $\bar{St}(v)$ for $v=(\alpha,2)$.}
\includegraphics[width=.4 \textwidth]{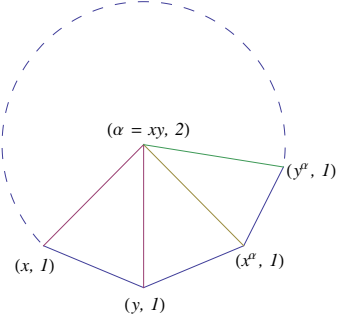}
\label{fig: twosheet}
\end{figure}

There are two possibilities: Either the $\alpha$-conjugacy orbits of these two triangles are distinct and the sheet is made from two $\alpha$-conjugacy orbits of triangles or
these two triangles are $\alpha$-conjugate and the sheet consists of a single $\alpha$-conjugacy orbit of triangles.

The latter happens if and only if $x$ and $y=x^{-1}\alpha$ are $\alpha$-conjugate.
This happens if and only if there is a positive integer $k$ such that
$x^{\alpha^k}=y$ i.e. $xx^{\alpha^k}=\alpha$.
The next lemma shows that this is equivalent to saying $\alpha$ is the product of two
distinct elements in the $\alpha$-conjugacy class of $x$.

\begin{lem}
Let $G$ be a finite group and $\alpha,x \in G$ be two elements which don't commute. Then the following are
equivalent: \\
(1) $xx^{\alpha^k}=\alpha$ for some positive integer $k$. \\
(2) $\alpha$ is the product of two distinct elements in the $\alpha$-conjugacy class of
$x$.
\end{lem}
\begin{proof}
Clearly (1) implies (2). Thus we only have to show (2) implies (1). Let $m$ be the size of the $\alpha$-conjugacy class of $x$, then $x^{\alpha^s}x^{\alpha^t}=\alpha$ for some
$0 \leq s,t < m$ with $s \neq t$. Conjugating this equation by $\alpha^s$ we find
$xx^{\alpha^{t-s}}=\alpha$. $t-s$ is congruent modulo $m$ to a unique positive integer $k$ with
$1 \leq k < m$ and $xx^{\alpha^k}=\alpha$. Thus we are done.
\end{proof}

\begin{defn}
Let $G$ be a finite group and $\alpha \in G$. There are three mutually exclusive, exhaustive possibilities
for an $\alpha$-conjugacy class $C$: \\
(1) $C$ has size $1$ and consists of a single element in $C(\alpha)$. \\
(2) $\alpha$ is the product of two distinct elements in $C$. We call such an
$\alpha$-conjugacy class productive. \\
(3) $\alpha$ is not the product of two distinct elements in $C$ and $C$ contains at least
$2$ elements. We call such an $\alpha$-conjugacy class nonproductive. \\
Let $p_{\alpha}$ denote the number of productive $\alpha$-conjugacy classes and
$n_{\alpha}$ denote the number of nonproductive $\alpha$-conjugacy classes.
Thus $p_{\alpha}+n_{\alpha}+|C(\alpha)|$ is the total number of $\alpha$-conjugacy classes.
\end{defn}

We now summarize the earlier analysis of the closed star of a type-2 vertex.

\begin{pro}[Closed star of type 2 vertices]
\label{pro: star2}
Let $\alpha$ be a noncentral element of a finite group $G$ and let $v=(\alpha,2)$ be the corresponding type 2 vertex in $X(G)$. Then for each sheet in $\bar{St}(v)$, the vertices of $Lk(v)$
in that sheet consist of either:\\
(1) Two non-productive $\alpha$-conjugacy classes of type 1 vertices. In this case
there are exactly two $\alpha$-conjugacy classes of triangles in that sheet and the two
$\alpha$-conjugacy classes have the same size. Thus the number of triangles in the sheet
is $2\ell$, an even number, where $\ell \geq 2$ is a divisor of the order of $\alpha$. \\
(2) A single productive $\alpha$-conjugacy class of type 1 vertices. In this case there is exactly one $\alpha$-conjugacy class of triangles in the sheet. Thus the number of triangles in the sheet is $\ell$ where $\ell \geq 3$ is an odd divisor of the order of $\alpha$. The number of triangles on the sheet in this case is odd.
\end{pro}
\begin{proof}
Most of this proposition was proven in the preceding paragraphs. Only a few final comments are in order. In case (1), the fact that the two orbits of triangles are interlaced shows that they have the same number of triangles which shows the two non-productive $\alpha$-conjugacy classes
in the link must have the same size. In case (2), the fact that a sheet must have at least $3$ triangles shows that
productive $\alpha$-conjugacy classes must have size $\geq 3$. (A direct algebraic argument can also show that the class cannot have size $2$ but we appeal to the geometry directly!)
The divisibility comments follow as the size of any $\alpha$-conjugacy orbit divides the order of $\alpha$. The fact that there are an odd number of triangles
in the case of one $\alpha$-conjugacy class follows because the action of $\alpha$-conjugacy always moves one $2$ steps along the rim of a $n$-triangle sheet,
thus there are two or one $\alpha$-conjugacy classes in a sheet, depending if the subgroup generated by $2$ has index two or one in the cyclic group of order $n$ which depends exactly if $n$ is even or odd respectively.
\end{proof}

Note from proposition~\ref{pro: star2}, whether a sheet about a type 2 vertex consists of two or one $\alpha$-conjugacy orbits, depends completely on whether it consists of an even or odd
number of triangles.

\begin{cor}
\label{cor: star2}
If $\alpha$ is a non-central element of a finite group $G$ then: \\
(1) $n_{\alpha}$ is even. \\
(2) The closed star of $(\alpha,2)$ in $X(G)$ consists of $\frac{n_{\alpha}}{2}+p_{\alpha}$ sheets.\\
(3) If $\alpha$ has prime order $p$, each sheet in the closed star of $(\alpha,2)$ has either
$p$ or $2p$ triangles depending on if it corresponds to a productive orbit or two non-productive orbits respectively. \\
(4) If $\alpha$ has order $2$ then each sheet in the closed star of $(\alpha,2)$ has $4$ triangles
and consists of two non-productive orbits. In particular $\alpha$ has no productive orbits and
$|G| \equiv |C(\alpha)| \text{ mod } 4$.\\
\end{cor}
\begin{proof}
Given $x$ that does not commute with $\alpha$ there is a unique $y$ such that $xy=\alpha$.
This implies that $(x,1)$ is on a unique sheet of $\bar{St}((\alpha,2)$. The vertices in the link of this sheet then consist of the $\alpha^{-1}$-orbit of $(x,1)$ and $(y,1)$ and hence consist of exactly the union of the two (not necessarily distinct) $\alpha$-conjugacy classes of $x$ and $y$.

By Proposition~\ref{pro: star2}, the non-productive $\alpha$-conjugacy classes must pair up
to make sheets and so there are an even number of them. Thus (1) is proven. (2) also follows immediately from the same proposition. (3) follows similarly once one notes that the only $\ell \geq 2$ dividing a prime $p$ is $\ell=p$. (4) follows as we have seen that productive $\alpha$-conjugacy classes must have size $\geq 3$ but also size dividing the order of $\alpha$ which is 2
and is hence impossible in this case. $G-C(\alpha)$ then decomposes as an even number of
non-productive orbits each of size 2 which implies $4$ divides $|G|-|C(\alpha)|$.
\end{proof}

Now we consider the closed star of vertices of type 1. Let $w=(\alpha,1)$ be a type 1 vertex of $X(G)$. Let $[(x,1),(\alpha,1),(x\alpha,2)]$ be a face contaning $w$. $[(\alpha,1),(x^{\alpha},1),(x\alpha,2)]$ and $[(\alpha,1),(x,1),(\alpha x,2)]$ are the adjacent faces containing $w$ also.
As $(\alpha x)^{\alpha}=x\alpha$, it is easy to check that the 2-labelled vertices in the part of
$Lk(w)$ in this sheet consists of the $\alpha$-conjugacy orbit of $\alpha x$. It also follows  that the 1-labelled vertices in the part of $Lk(w)$ in this sheet consists of the $\alpha$-conjugacy orbit of $x$. The $\alpha$-conjugacy orbit of the $1$-vertices in the link for this sheet has the same size as the $\alpha$-conjugacy orbit of the $2$-vertices in the link as the edges along the link of $w$ alternate between type 1 and type 2 vertices.
Furthermore the $\alpha$-conjugacy class of the $2$-vertices is uniquely determined from the one for the $1$-vertices by the fact that there exist $u$ in the type 2 vertex orbit and $v$ in the type 1 vertex orbit whose "difference" $v^{-1}u$ is $\alpha$.

We have thus proven:

\begin{pro}[Closed star of type 1 vertices]
\label{pro: star1}
Let $G$ be a nonabelian group and let $w=(\alpha,1)$ be a type $1$ vertex in $X(G)$.
Then each sheet of $\bar{St}(w)$ has $2\ell$ triangles where $\ell \geq 2$ divides the order of
$\alpha$. The vertices along $Lk(w)$ in any sheet alternate between type 1 and type 2 vertices and there is a single $\alpha$-conjugacy orbit of type 1 vertices and a single $\alpha$-conjugacy orbit of type 2 vertices in the link of any given sheet. These orbits have the same size. One orbit determines
the other by the fact the type 2 orbit contains $u$ and the type 1 orbit contains $v$ such that $v^{-1}u=\alpha$. As each sheet contains one $\alpha$-conjugacy class of size $> 1$ each of type 1
and type 2 vertices, the total number of sheets in $\bar{St}(w)$ is the total number of
$\alpha$-conjugacy classes of size $> 1$ i.e., $n_{\alpha}+p_{\alpha}$.
\end{pro}

\begin{defn}
Let $G$ be a nonabelian group and $v$ a vertex of $X(G)$. Let $s(v)$ be the number of sheets
the closed star of $v$ has and let $V$ denote the vertex set of $X(G)$. \\
(1) If $v=(\alpha,2)$ then $s(v)=\frac{n_{\alpha}}{2}+p_{\alpha}$.\\
(2) If $v=(\alpha,1)$ then $s(v)=n_{\alpha} + p_{\alpha}$. \\
(3) $\sum_{v \in V} s(v) = \sum_{\alpha \in G} (\frac{3n_{\alpha}}{2}+2p_{\alpha})$.
\end{defn}

We are now ready to present another formula for the Euler characteristic of $X(G)$.

\begin{thm}
Let $G$ be a nonabelian group then
$$
\chi(X(G))=2(|G|-|Z(G)|)-\sum_{\alpha \in G} (\frac{3n_{\alpha}}{2}+2p_{\alpha}) + \sum_{g \geq 0} (2-2g)m_g(G).
$$
Combining this with previous formulas for $\chi(X(G))$ yields the formula:
$$
\sum_{\alpha \in G}(3n_{\alpha}+4p_{\alpha}) + \sum_{g \geq 0} (4g-4)m_g(G) = |G|(|G|-c)
$$
where $c$ is the number of conjugacy classes of $G$.
\end{thm}
\begin{proof}
Recall that $Y(G)$, the desingularization of $X(G)$ is the disjoint union of $m_g(G)$ Riemann
surfaces of genus $g$ over all $g \geq 0$. Thus $\chi(Y(G))=\sum_{g \geq 0} (2-2g)m_g(G)$.
Recall that $Z(G)$ was a complex obtained by attaching disjoint closed star graphs to $Y(G)$
where there was one such graph for each vertex in $X(G)$ and the type of star graph used for any given vertex depended on the number of sheets in the closed star of that vertex.
This means $Z(G)$ is obtained from $Y(G)$ by adding $|V|$ vertices where $V$ is the vertex set  of $X(G)$ and by adding $\sum_{v \in V} s(v)$ edges. This means
$\chi(Z(G)) = \chi(Y(G)) - \sum_{v \in V} s(v) + |V|$. As $X(G)$ is homotopy equivalent to $Z(G)$ we have $\chi(X(G))=\chi(Z(G))$. Combining this with the fact that $|V|=2(|G|-|Z(G)|)$ and
putting everything together we get the formula claimed for $\chi(X(G))$. Comparing this
to previous formulas for $\chi(X(G))$ yields the second formula immediately.

\end{proof}

\subsection{Functoriality}

A group homomorphism $f: G \to H$ is said to be injective on commutators
if the restriction $f: G' \to H$ is injective where $G'$ is the commutator subgroup of $G$.
It is easy to check that the composition of two homomorphisms which are injective on commutators is also injective on commutators. If $f: G \to H$ is such a homomorphism then
$g_1$ commutes with $g_2$ if and only if $f(g_1)$ commutes with $f(g_2)$ for all
$g_1, g_2 \in G$.

Let $\mathfrak{C}$ denote the category of finite groups and homomorphisms
which are injective on commutators. Let $\mathfrak{D}$ be category of finite oriented $2$-dimensional
simplicial complexes and orientation preserving simplicial maps. Here by an oriented
$2$-dimensional simplicial complex we mean one whose faces ($2$-simplices) have all been given an orientation and when we say a simplicial map preserves orientation we mean it takes
$2$-simplices to $2$-simplices in a manner that preserves orientation of the individual
faces.

\begin{pro}
The construction $X(G)$ is part of a covariant functor from the category $\mathfrak{C}$
of finite groups and homomorphisms injective on commutators to the category $\mathfrak{D}$
of finite oriented $2$-dimensional simplicial complexes and orientation preserving simplicial maps. In particular if $G_1$ is isomorphic to $G_2$
then $X(G_1)$ is simplicially isomorphic to $X(G_2)$ and if $H \leq G$ then $X(H)$ is a subcomplex of $X(G)$.
\end{pro}
\begin{proof}
We have already described $X$ on the level of objects so let $f: G \to H$ be a homomorphism
injective on commutators. For vertices (either of type 1 or type 2) define
$X(f)((v,i))=(f(v),i)$ for $i=1,2$. As $f$ takes noncommuting elements to noncommuting
elements, it takes edges of $X(G)$ to those of $X(H)$. If $[(x,1), (y,1), (xy,2)]$ is an oriented
face of $X(G)$, then \\ $[(f(x),1),(f(y),1),(f(xy),2)]$ is an oriented face of $X(H)$ as
$f(xy)=f(x)f(y)$. Thus $X(f)$ defines an orientation preserving simplicial map between
$X(G)$ and $X(H)$. It is now easy to check that $X$ respects compositions and identity maps
and defines a covariant functor from $\mathfrak{C}$ to $\mathfrak{D}$ as desired. The rest
follows readily.
\end{proof}

\begin{cor}
For $G$ a finite nonabelian group, $Aut(G)$ acts on $X(G)$ through orientation preserving
simplicial automorphisms. In particular $G$ acts on $X(G)$ simplicially by conjugation.
Furthermore an anti-automorphism of $G$ like $\theta(g)=g^{-1}$ induces an orientation
reversing simplicial automorphism of $X(G)$.
\end{cor}

The covariant functor $X$ induces another functor $Y$, the desingularization of $X$.
We have already described the desingularization $Y(G)$ on the level of objects.
For a homomorphism $f: G \to H$ which is injective on commutators, we have already defined an orientation preserving simplicial map $X(f): X(G) \to X(H)$. As such a map takes faces to faces,
it is easy to see that it takes pseudomanifold components to pseudomanifold components and induces a continuous map $X(G)-\{ \text{vertices}\} \to X(H)-\{ \text{vertices} \}$ which
takes punctures to punctures. As each puncture arises from a unique sheet of the closed star
of a unique vertex, or equivalently from a unique circle of the link of a unique vertex, to see if there is a well-defined continuous extension of $X(f)$ to a simplicial map $Y(G) \to Y(H)$, we need only note that each circle in a link of the vertex $(v,i)$ must map to a unique circle
in the link of the vertex $(f(v),i)$ under $X(f)$. We then define $Y(f)$ in such a way as to map
the puncture associated to  a particular circle in the link of $(v,i)$ in $X(G)$ to the puncture
associated to the circle in the link $(f(v),i)$ in $X(H)$ which $X(f)$ takes the first circle to.

After doing this, it is easy to check $Y(f)$ is a well-defined orientation preserving simplicial map
from $Y(G)$ to $Y(H)$ and that the construction preserves compositions and identity maps. Thus we have proven:

\begin{pro}
The construction $Y(G)$ is part of a covariant functor $Y$ from $\mathfrak{C}$ to
$\mathfrak{D}$.
\end{pro}

From this it follows that if $H$ is a subgroup of $G$, then the manifold components of
$Y(H)$ are a subset of those of $Y(G)$. Thus $m_g(H) \leq m_g(G)$ for every genus $g$.
Thus we have proven:

\begin{cor}[Monotonicity]
Let $H \leq G$ be finite non-abelian groups, then $m_g(H) \leq m_g(G)$ for every genus $g \geq 0$.
\end{cor}

The next corollary follows from monotonicity as every finite group is a subgroup of a symmetric group
$\Sigma_n$. It shows that a particular genus $g$ surface can occur
as a component of $Y(G)$ for finite groups if and only if it can occur for symmetric groups.

\begin{cor} If $m_g(G) > 0$ for some genus $g \geq 0$ and finite non-abelian group $G$
then $m_g(\Sigma_n) > 0$ for some $n \geq 3$.
\end{cor}

We also record the important fact that every component of $X(G)$ originates from a component of $X(H)$ where
$H \leq G$ is a non-abelian subgroup generated by 2 elements. In this regard much like Coxeter's hyperbolic Cayley graph construction,
the Riemann surfaces that make up
$X(G)$ originate from various $2$-generated subgroups. However the triangulations on these surfaces and method of generation seems to differ greatly so it is unclear what the exact relationship is between the two constructions.

\begin{cor}
Let $G$ be a finite non-abelian group, then $X(G) = \cup_{H \in \mathfrak{A}} X(H)$ where $\mathfrak{A}$ denotes the collection of subgroups of
$G$ which are non-abelian and generated by $2$-elements.
\end{cor}
\begin{proof}
Starting from a triangle in $X(G)$, say $[(x,1), (y,1), (xy,2)]$ it is easy to verify by induction that all triangles in the pseudomanifold component which contains that triangle have vertices which lie in the subgroup $H=<x,y>$ which is generated by $x$ and $y$. As $x$ and $y$ don't commute, $H$ is non-abelian and
generated by $2$-elements. As $X(G)$ is the union of its pseudomanifold components, the corollary is proved.
\end{proof}

\subsection{Abstract 3-polytope structure of components}

In this section, we discuss a modification of the cell structure of the triangulated Riemann surfaces that occur as components of $Y(G)$ and show this cell structure satisfies the conditions of what is called an abstract $3$-polytope in the combinatorics literature. Fix a nonabelian finite group $G$ in this section.

Recall, we have seen that the components of the desingularization $Y(G)$ are triangulated Riemann surfaces. The closed star of a type 2 vertex in this surface is a single sheet which is simplicially isomorphic to some simplicial $m$-disk.

Let $\hat{Y}(G)$ be the same collection of Riemann surfaces in $Y(G)$ except that the cell-structure is modified as follows. For each component, all type 2-vertices and edges joining type 1 to type 2 vertices are erased resulting in a new cell structure for the component where the vertices consist of only the type 1 vertices of the original construction, the edges only those joining type 1 vertices in the original construction and where the faces are polygonal consisting
of the sheet about a type 2 vertex in the original construction with its nonboundary vertices and edges erased. Topologically we are not changing the component at all, we are just changing the cell structure. The final cell structure will be called the $2$-cell structure of the component. It is no longer a triangulation but does give the component the structure of a $CW$-complex. This cell structure forms a ``closed cell structure", i.e., the closed cells are homeomorphic to $2$-disks and no boundary identifications occur between points on the boundary of the same face.

We first show that all $2$-sheets in a given component of $Y(G)$ are simplicially isomorphic
and hence all the $2$-faces in the corresponding cell structure of a given component of $\hat{Y}(G)$ consist of the same type of polygonal face.

\begin{thm} Let $G$ be a nonabelian finite group and let $T$ be the unique component of $Y(G)$ which contains the triangle $[(x,1), (y,1), (xy,2)]$. Then all type 1-vertices in $T$ are conjugate to either $x$ or $y$ and all type 2-vertices in $T$ are conjugates of $xy$. Furthermore all sheets about type $2$-vertices in $T$ are simplicially isomorphic via a conjugation and hence are $n$-gons for the same $n \geq 3$. These conjugations are by elements in the subgroup $H$ generated by $x$ and $y$.

In the corresponding $2$-cell structure, all vertices are conjugate to (and hence have the same valency as) either $(x,1)$ or $(y,1)$. All faces are $n$-gons for the same $n$. All oriented edges are conjugate to the edge $[(x,1),(y,1)]$ or the edge $[(y,1),(y^{-1}xy,1)]$ and all unoriented
edges are conjugate. Thus the cell automorphism group of the component is both face and edge transitive.
\label{thm: conjugate1}
\end{thm}
\begin{proof}
Let $\alpha=xy$. A quick calculation shows that the three triangles adjacent to $[(x,1),(y,1),(\alpha,2)]$ in $T$ are
$[(y,1),(x^{\alpha},1),(\alpha,2)], [(y,1),(x,1), (\alpha^{x},2)]$ and $[(xyx^{-1},1),(x,1),(\alpha,2)]$.
The vertices in these 3 triangles are conjugate to those in the original triangle using conjugation by elements in the group $H$ generated by $x$ and $y$. As every triangle in the component can
be joined to the original one by a sequence of adjacent triangles, and the union of the triangles is the component, the conjugacy result for vertices follows.

Now consider the sheet about $(\alpha,2)$ which contains the triangle \\ $[(x,1),(y,1),(\alpha,2)]$
and the adjacent sheet that shares the edge $[(x,1),(y,1)]$ i.e. the sheet about $(yx,2)$ which contains the
triangle $[(y,1),(x,1),(yx,2)]$. Conjugating the triangle $[(x,1),(y,1),(\alpha,2)]$ by $x^{-1}$ and using functoriality yields a simplicial isomorphism of $Y(G)$ which takes the triangle to the one
$[(x,1),(x^{-1}yx,1), (yx,2)]$ which is adjacent to $[(y,1),(x,1),(yx,2)]$. From this it follows that this simplicial automorphism must take the component $T$ back to itself and the sheet about
$(\alpha=xy,2)$ to the one about $(yx,2)$ in the same component. In particular two adjacent sheets in a given component are conjugate and hence simplicially isomorphic. As any sheets about type 2-vertices in a given component can be joined by a sequence of adjacent sheets, the result on sheets follows.

To get the result on edges, first note every edge is conjugate to one that lies in the original sheet. Then previous results show that under conjugation by the center element of the sheet there are at most two orbits of oriented edges depending if the conjugacy classes $x$ and $y$ are a pair of nonproductive $\alpha$-conjugacy classes or a single productive one. These (unoriented) edges are represented by orbit representatives $[(x,1),(y,1)]$
and $[(y,1), (y^{-1}xy,1)]$ in the case of two $\alpha$-orbits. Conjugation by $y^{-1}$ takes the first edge to the second (with orientation flipped) and is easily seen to induce an automorphism of the component. Thus the cell automorphism group of the component is (unoriented) edge transitive and the theorem is proven.

\end{proof}

\begin{cor}
Let $G$ be a finite nonabelian group then the number of components of $Y(G)$ is greater or equal to the number of conjugacy classes of noncentral elements in $G$.
\end{cor}
\begin{proof}
By Theorem~\ref{thm: conjugate1}, there is at most one conjugacy class represented by the type 2 vertices in a given component. As there has to be at least one type 2 vertex for each noncentral element of $G$, the corollary follows. (Note the desingularization process can cause there to be more than one type 2-vertex corresponding to a given noncentral element of $G$. Also examples show a conjugacy class can be spread out over more than one component. Both these factors cause the inequality just proved to often not be equality.)

\end{proof}

Recall the modified cell structure of components denoted $\hat{Y}(G)$ where the vertices are only the type 1 vertices, the edges only those joining two type 1 vertices and $2$-faces being $n$-gons which were sheets about type 2 vertices in $Y(G)$. These $n$-gons will carry an implicit label given by the original middle element of the corresponding sheet. Theorem~\ref{thm: conjugate1} shows that a given component with this cell structure has face and edge-transitive cell automorphism group (bijections of vertices which carry edges to edges and boundaries of $2$-faces to boundaries of $2$-faces) and either one or two vertex orbits.
The next theorem collects important formulas for the quantities in this component.

\begin{thm}
\label{thm: conjugate2}
Let $G$ be a finite nonabelian group and let $T$ be the unique component of $\hat{Y}(G)$ which corresponds to the triangle $[(x,1),(y,1),(\alpha=xy,2)]$ where $x,y \in G$ do not commute.
$T$ is a compact, connected, oriented $2$-manifold of genus $g$ with cell structure consisting
of $2$-faces which are all $n$-gons for some fixed $n \geq 3$. The cell automorphism
group always acts face and (unoriented) edge transitively on $T$. \\
(1) If $x$ and $y$ lie in different $\alpha$-conjugacy classes ($n$ even case), then $n=2(\text{size of }\alpha-\text{conjugacy class of } x)$ and the automorphism group acts with at most $2$-orbits of vertices represented by $(x,1)$ and $(y,1)$ respectively. The valency of $(x,1)$ in $T$ with the $2$-cell structure is
given by $\lambda_1 \geq 2$, the size of the $x$-conjugacy class of $\alpha=xy$. The valency of $(y,1)$ in $T$ in this modified $2$-cell structure is $\lambda_2 \geq 2$, the size of the $y$-conjugacy class of $\alpha=xy$.
In the case $\lambda_1 \neq \lambda_2$ we define $\lambda=\frac{\lambda_1V_1 + \lambda_2V_2}{V_1+V_2} \geq 2$ to be the average valency of the component where $V_i$ denotes the number of vertices of valency $\lambda_i$ in the component. In this case each edge in the component joins a vertex of valency
$\lambda_1$ with a vertex of valency $\lambda_2$ and we also have $E=\lambda_iV_i$ for $i=1,2$. Finally the average valency can be computed as either $\lambda= \frac{2E}{V}$ or as the harmonic average of the valencies $\lambda_1, \lambda_2$, i.e.,
$$
\frac{2}{\lambda} = \frac{1}{\lambda_1} + \frac{1}{\lambda_2}.
$$
(2) If $x$ and $y$ lie in the same $\alpha$-conjugacy class ($n$ odd case), then $$n=\text{size of }\alpha-\text{conjugacy class of } x$$ and the automorphism group acts transitively on edges and vertices also.
The common valency $\lambda \geq 2$ of all vertices is given by the size of the $x$-conjugacy class of $\alpha$. \\
(3) If $V, E, F$ denote the number of vertices, edges and $2$-faces of the $2$-cell structure of the component $T$ and the face type of the component is $n$-gons, with average vertex valency $\lambda$ then the following equations hold: \\
$$
nF=2E
$$
$$
nF=\lambda V
$$
$$
2-2g=V-E+F=2(\frac{1}{\lambda}+\frac{1}{n}-\frac{1}{2})E
$$
\end{thm}
\begin{proof}
The formulas for $n$ follow from results in previous sections concerning the number of triangles in a sheet about a type $2$-vertex. The formulas for valency follow once one notes
that the sheet centered about $(xy,2)$ and the one about $(yx,2)$ containing $(x,1)$ are adjacent (share an edge) and are conjugate to each other by conjugation by $x^{-1}$. Repeating this observation, one finds that the $2$-faces containing $(x,1)$ in $T$ consist of those centered
at the $x$-conjugacy orbit of $(\alpha=xy,2)$. Thus the number of faces containing
$(x,1)$ which is the same as the number of edges incident to $(x,1)$ in the $2$-cell structure
is given by the size of this $x$-conjugacy class. Similar arguments work for the vertex $(y,1)$.
As the cell automorphism group of the component is edge transitive, every edge joins a vertex conjugate to $(x,1)$ to a vertex conjugate to
$(y,1)$ where the edge $[(x,1),(y,1)]$ is the original one determining the component. Thus each edge joins a vertex of valency $\lambda_1$ to one of valency $\lambda_2$. When $\lambda_1 \neq \lambda_2$ note that each edge contributes a total of one to the valency count of
vertices of valence $\lambda_1$. As $\lambda_1$ of these edges are incident on a given such vertex we find $E=\lambda_1V_1$.
Similarly $E=\lambda_2V_2$. Note also that $V=V_1+V_2$. Now it follows that $2E=\lambda_1V_1+\lambda_2V_2 = \lambda V$ by definition
so $\lambda=\frac{2E}{V}$. Now note:
$$
\frac{1}{\lambda_1} + \frac{1}{\lambda_2} = \frac{V_1}{E} + \frac{V_2}{E} = \frac{V}{E}=\frac{2}{\lambda}
$$
and so $\lambda$ is the harmonic average of $\lambda_1$ and $\lambda_2$.

Finally the first formula in $(3)$ follows from the observation that each polygonal face contributes $n$ edges to the component and each edge lies in exactly $2$ such faces. We will prove the second formula in the harder case of two types of vertex valencies - the other case follows similarly. First recall that each edge will contribute two vertices, one of which is in the orbit of $(x,1)$ and the other in the orbit of $(y,1)$. Now each face contributes $n$ vertices with an equal number of valency $\lambda_1$ as with valency $\lambda_2$. Again taking account that a vertex of valency $\lambda_i$ lies in $\lambda_i$ such faces we get $\frac{nF}{2} = \lambda_iV_i$. Adding these equations for $i=1,2$ yields $nF=\lambda_1V_1+\lambda_2V_2=\lambda V$ and so the second formula is proven. Finally the Euler characteristic of a Riemann surface of genus $g$ is $2-2g$ and equals the alternating sum $V-E+F$ in any cell decomposition. Thus the final formula follows upon plugging in the previous formulas.
\end{proof}

It follows from face and edge transitivity that many of the quantities in the last theorem have stringent divisibility conditions. The next proposition
records these:

\begin{pro}
\label{pro: divisibility}
Let $G$ be a finite nonabelian group and let $T$ be a component of $\hat{Y}(G)$ with the $2$-cell structure discussed in Theorem~\ref{thm: conjugate2} consisting of $n$-gon faces and with vertex, edge and face counts $V, E, F$ respectively. Then: \\
(1) $E \geq 3$ and $F \geq 3$ are divisors of $|G|$. \\
(2) In the case of two distinct vertex orbits $V_1, V_2, \lambda_1, \lambda_2 \geq 2$ are divisors of $E$ and hence of $|G|$. Furthermore $n \geq 4$ is even and divides $2E$ and hence $2|G|$. \\
(3) In the case of one vertex orbit $\lambda\geq 2, V, n\geq 3$ are divisors of $2E$ and hence of $2|G|$. \\
(4) If either $E$ or $F$ is equal to $|G|$ then the component $T$ must be invariant under conjugation by $G$ and $G$ must have trivial center.
In fact if the edge $[(x,1),(y,1)]$ lies in the component then $C(x) \cap C(y) = \{ 1 \}$. \\
(5) There are finitely many possibilities for all the data of the component \\
$(g,V_i,E,F,\lambda_i, n)$ given $|G|$ determined by these simple divisibility conditions.
\end{pro}
\begin{proof}
$G$ acts on $\hat{Y}(G)$ cellularly by conjugation. Thus it shuffles components around taking components to isomorphic components.
The $G$-conjugacy orbit of the component $T$ is thus a union of $\ell \geq 1$ components all isomorphic as cell complexes with $T$.
Taking an edge in $e$ in the component $T$, Theorem~\ref{thm: conjugate1} shows that the $G$-conjugation orbit of $e$ includes all the edges in $T$. It is then easy to see that the $G$-conjugation orbit of $e$ is exactly the set of edges in the $\ell$ components conjugate to $T$.
Since each of these components has the same edge count $E$ we have $\ell E = \frac{|G|}{|S|}$ where $S$ is the stabilizer subgroup of the edge $e$. Thus $|G|=\ell E |S|$ and so $E$ divides $|G|$. Furthermore $E=|G|$ if and only if $|S|=1$ and $\ell=1$ i.e. $T$ is invariant
under $G$-conjugation and $G$ acts freely and transitively on the set of edges of $T$. As any element of $C(x) \cap C(y)$ fixes the edge
$[(x,1),(y,1)]$ under conjugation, we must have $C(x) \cap C(y) = \{ 1 \}$ in this case and in particular $G$ must have trivial center.
An analogous argument works for faces and so (1) and (4) are proved.  In case (2), we have $E=\lambda_i V_i$ and $2E=nF$ and so the result
follows. Also $n$ is even in this case as the proof of Theorem~\ref{thm: conjugate2} shows. In case (3), we have $2E=\lambda V=nF$ and that case follows immediately also. Finally (5) follows as there are only finitely
many divisors of a positive integer $|G|$ and $2-2g=V_1+V_2-E+F=V-E+F$ is an ``alternating" sum of three or four of these divisors.
\end{proof}

The reader is warned that in general $V=V_1 + V_2$ does not divide $|G|$ as examples in later sections show.

The $2$-cell structures on the components of $\hat{Y}(G)$ can be modified trivially to form what is called an abstract $3$-polytope in the combinatorics literature. The definition of these objects is a bit long so we will not repeat it here but the reader may find it in \cite{Cox2} and \cite{MS}. We will verify these conditions briefly here.

First note the cell structure on a given component $T$ of $\hat{Y}(G)$ can be used to form a
poset consisting of the set $F_{-1} \cup F_0 \cup F_1 \cup F_2 \cup F_3$ where $F_k$ will be
the set of $k$-faces in the poset ordered by inclusion. Here $F_0$ is the set of vertices of the component, $F_1$ is the set of edges of the component and $F_2$ is the set of polygonal $2$-faces of the component. $F_{-1}$ consists of a single empty face, which is the smallest element of the poset and $F_3$ consist of a formal "biggest" face containing all other faces.

A flag in this poset consists of $\emptyset \subset v \subset e \subset f \subset D$
where $v$ is a vertex, $e$ is an edge, $f$ is a $2$-face and $D$ is the formal greatest $3$-face.
All such flags have length $5$.

This poset satisfies the ``diamond condition" of a polytope as every edge contains exactly two
vertices, every edge is contained in exactly two $2$-faces and for any $v \in f$ there are exactly
two edges $e,e'$ such that $v \in e \subset f, v \in e' \subset f$.

Finally the poset is strongly connected as each component is a path-connected Riemann surface with path connected $2$-faces as the reader can verify.

Thus the $2$-cell structure of any component extended formally to a poset with smallest emptyset face and greatest $3$-face satisfies the properties of an abstract $3$-polytope.
The automorphism group of this poset is easily seen to be the same as the group of bijections of the vertex set to itself which takes edges to edges and faces to faces i.e. is the same as the group of cell automorphisms of the $2$-cell complex $T$. We have already seen this automorphism group is always face and edge transitive. This translates to the sections $S(\emptyset, f)$
of this poset to be all $n$-polygons for fixed $n \geq 3$. There are at most 2 orbits of vertices,
the sections $S(v,D)$ thus are polygons of at most two different types.

To be equivar in the language of polytopes, these valencies must be the same. In addition to be regular, the automorphism group must act transitively on flags. We identify conditions for these situations in the next "dictionary" theorem:

\begin{thm}
\label{thm: polytope}
Let $G$ be a nonabelian group and $T$ the unique component of $\hat{Y}(G)$ corresponding to the triangle $[(x,1),(y,1),(xy,2)]$ with the extended cell structure discussed in the previous paragraphs. Recall all $2$-cells of $T$ are $n$-gons. Then: \\
(1) This cell structure of $T$ forms an abstract 3-polytope with $2$-face and edge transitive automorphism group. \\
(2) The abstract $3$-polytope is equivar if and only if all vertices have the same valency $\lambda$.
This happens when the size of the $x$ and $y$-conjugacy classes of $\alpha=xy$ are the same.
The Schl\"afli index of this equivar 3-polytope is $\{n,\lambda\}$. \\
(3) The abstract $3$-polytope has $2$-face, edge and vertex transitive automorphism group if $x$ and $y$ are
$\alpha$-conjugate. \\
(4) The abstract $3$-polytope has flag-transitive automorphism group (i.e. is regular) if $x$ and $y$ are $\alpha$-conjugate and there is an automorphism of the group $G$ taking $x$ to $y^{-1}$
and $y$ to $x^{-1}$ \\
(5) The results of (3) and (4) still hold when $x$ and $y$ are not $\alpha$-conjugate as long as there exists
an automorphism of the group $G$ taking $x$ to $y$ and $y$ to $y^{-1}xy$ and another taking $x$ to $y^{-1}$ and
$y$ to $x^{-1}$.
\end{thm}
\begin{proof}
(1) and (2) follow from the work in Theorem~\ref{thm: conjugate1} and the previous paragraphs.
To see (3), note that any edge or vertex of the component can be mapped by an automorphism
to lie in the $2$-cell centered at $(\alpha,2)$. Then $\alpha$-conjugacy will move this image edge or vertex to any chosen representative edge or vertex on that $2$-cell as long as all vertices on the boundary of this $2$-cell are $\alpha$-conjugate. This happens if and only if $x$ and $y$
are $\alpha$-conjugate. For (4), given a reference flag $v \in e \subset f$ (we will surpress the empty face and greatest face in this proof as they do not come into play) and another flag $v' \in e' \subset f'$ we can first apply an automorphism to take $f'$ to $f$ by face transitivity. Thus we many assume $f=f'$. Then as all vertices along the rim of the corresponding $2$-cell are conjugate, we may conjugate fixing $f$ so that $e'$ moves to $e$. Thus it remains to show
that there is a automorphism of the polytope taking the flag $(x,1) \in [(x,1),(y,1)] \subset [(x,1),(y,1),(xy,2)]$ to the flag $(y,1) \in [(x,1),(y,1)] \subset [(x,1),(y,1),(xy,2)]$. As such a map has to be orientation reversing, to achieve it using functoriality we have to use an anti-automorphism of the group. If $I$ is the inversion map $I(x)=x^{-1}$, then any anti-automorphism is the composition of $I$ with an automorphism. A quick calculation shows that if $\phi$ is an automorphism of $G$ taking $x$ to $y^{-1}$ and $y$ to $x^{-1}$, the simplicial automorphism arising from $I \circ \phi$ does the job! Finally for (5), the stated automorphism maps the edge $[(x,1),(y,1)]$ to its adjacent edge on the sheet about
$(xy,2)$. This together with previous comments gives vertex transitivity. Then regularity follows as in the proof of (4).
\end{proof}

The duality operation of reversing the poset ordering in the poset underlying an abstract $3$-polytope has the effect of
interchanging the roles of vertices and $2$-faces while leaving the edges alone. It takes an equivar $3$-polytope of Schl\"afli symbol $\{ n, \lambda \}$ to one of symbol $\{ \lambda, n \}$. It corresponds to the classical dual cell structure construction behind Poincare duality of the component. The vertices in this new structure correspond to the centers of the $2$-faces in the original. Edges are drawn between these vertices when the $2$-faces they came from were adjacent. The new $2$-faces hence come from arrangements of faces around vertices in the original and are $\lambda_i$-gons
if the vertex had valency $\lambda_i$. Thus the dual of one of our complexes would be vertex and edge transitive and have either one or two orbits of faces. In the case of two orbits, the faces would consist of either $\lambda_1$-gons or $\lambda_2$-gons  and around every vertex
these types would alternate with total even vertex valency $n$. Thus all in all, the role of face type and valency interchange and the role of
vertices and faces change under duality. Thus if the reader prefers, they can consider the dual to our construction which would be like a "soccerball", consisting of at most two types of polygonal faces with vertex and edge transitive automorphism group. In other words the complexes that arise as components in the construction $\hat{Y}(G)$ that have two valencies are in general dual to abstract quasiregular (i.e., vertex, edge transitive with two types of faces arranged alternatingly about a vertex) $3$-polytopes.

\subsection{Valence Two and Doubling}

In this section we will discuss the case when one or both valencies in Theorem~\ref{thm: conjugate2} are equal to two.

\begin{lem}
\label{lem: valence2}
Let $X$ be a tesselated Riemann surface as those arising
is Theorem~\ref{thm: conjugate2} i.e., edge and face transitive and with at most two orbits of vertices. If the valencies $\lambda_1=\lambda_2=2$ then $F=2, E=V=n, g=0$ and $X$ is a sphere obtained by gluing two $n$-gon faces along their common rim.
\end{lem}
\begin{proof}
$\lambda_1=\lambda_2=2$ implies $\lambda=2$. Using this in the equations in part (3) of Theorem~\ref{thm: conjugate2}, we find $2-2g=\frac{2E}{n}$. As this quantity is positive, this forces $g=0$ and then $E=n$. Then $nF=2E$ forces $F=2$ and $\lambda V = 2E $ forces $E=V$. The lemma follows.
\end{proof}

We now consider the case of a tesselated Riemann surface which is edge/face transitive
and has two orbits of vertices of valency $\lambda_1=2$ and $\lambda_2=k > 2$. Recall when we have two valencies we have an even face type $2s$. We will denote the Schl\"afli symbol of such a complex as $\{2s, 2\text{-}k\}$ where $2s$ is the face type.

If we have such a complex, the two orbits of vertices are distinguishable due to their distinct valencies. Notice each vertex of valency $2$ lies in two faces and is adjacent to two vertices of valency $k > 2$ which also lie in both of these faces. Thus we can remove all valency $2$ vertices and combine the two incident edges to any of them into a single edge. This construction doesn't change the underlying surface so the genus is unchanged. The face type changes from $2s$-gon to $s$-gon and the new edge count is half the original one. The new vertex count is equal to the count of valency $k$ vertices in the original complex. The resulting complex is equivar with a single vertex valency $k$. Thus we have changed a $\{2s,2\text{-}k\}$ complex into an equivar
$\{s,k\}$ complex. Conversely given an equivar complex of the form $\{s,k\}$ one can add a midpoint vertex to each edge to obtain a $\{2s,2\text{-}k\}$ one and it is easy to see these processes are inverse processes.

We will refer to the $\{2s,2\text{-}k\}$ complex obtained from the $\{s,k\}$ one as the ``double" of the $\{s,k\}$-complex. We will sometimes write $\{2s,2\text{-}k\}=D\{s,k\}$. Although the complexes are so similar, it is important to note that if both occur in the decomposition $Y(G)$ for a given group $G$, they are distinct functorially, i.e., no group automorphism can interchange the two types. Also note that $\{4,2\text{-}k\}=D\{2,k\}$; in this case only, after removing valency two vertices, one obtains an equivar complex whose faces are $2$-gons. While the $\{2,k\}$ complexes do not arise in $\hat{Y}(G)$ since $n<3$, the $\{2,k\}$ complexes are duals of the $\{k,2\}$ complexes described in Lemma~\ref{lem: valence2}. Since the doubling and duality operations are genus preserving, the $\{4,2\text{-}k\}$ complexes have genus zero ($g=0$).

\subsection{A Finiteness Theorem}

In this section we show that except the genus one ($g=1$), and the $\{n,2\}$ and $\{4,2\text{-}k\}$ families described above, for a given genus there are only finitely many distinct tesselations on the closed surface of genus $g$ of the sort arising in Theorem~\ref{thm: conjugate2}. Recall that these tesselations are closed cell structures, i.e., the closed cells are homeomorphic to $2$-disks, that is to say, no self-identifications occur along the boundary of the faces.

\begin{thm}
\label{thm: finiteness} Let $g, V, E, F, n, \lambda_1, \lambda_2$ be as in Theorem~\ref{thm: conjugate2}. The distinct closed-cell tesselations on the closed surface of genus $g$ which are edge and face transitive having $n\geq 3$, can be categorized as follows:\\
(i) For each fixed genus $g\geq 2$, there are only finitely many possibilities for all the data $(V_i, E, F, n, \lambda_i)$.\\
(ii) For $g=1$, there are only finitely many possibilities for the Schl\"afli symbol $\{n,\lambda\}$ or $\{n,\lambda_1\text{-}\lambda_2\}$. There are infinitely many possibilities for $V_i, E,$ and $F$.\\
(iii) For $g=0$, there are infinite families when $\lambda_1=\lambda_2=2$ and $\{n,\lambda_1\text{-}\lambda_2\}=\{4,2\text{-}k\}, k\geq 3$. Otherwise, there are only finitely many possibilities for all the data $(V_i, E, F, n, \lambda_i)$.
\end{thm}
\begin{proof}

The $\{n,2\}$ (i.e. $\lambda_1=\lambda_2=2$) and $\{4,2\text{-}k\}$ cases are described by Lemma~\ref{lem: valence2} and the discussion that followed. Furthermore, the doubling operation shows that the remaining valence two cases where $\lambda_1=2$ and $\lambda_2=k>2$, are in one-to-one correspondence with the equivar complexes $\{s,k\}$, where $s\geq 3$. Hence, we need only prove the finiteness assertions for $\lambda_1, \lambda_2 \geq 3$.

We utilize the equations presented in Theorem~\ref{thm: conjugate2}:
\begin{eqnarray}
nF&=&2E \label{eq: 1}
\\
nF&=&\lambda V \label{eq: 2}
\\
2(1-g)=V-E+F&=&2(\frac{1}{\lambda}+\frac{1}{n}-\frac12)E, \label{eq: 3}
\end{eqnarray}
where $\lambda=\lambda_1=\lambda_2$ is the common valency in the equivar case, and $\frac{1}{\lambda}=\frac{1}{2\lambda_1}+\frac{1}{2\lambda_2}$ in the case of two valencies. In general $n\geq 3$, and in the two valency case $n$ must be even (hence $n\geq 4$).

Let $g\geq 1$. Using $n\geq 3$ in Equation \ref{eq: 3} we have
\begin{equation}
E(\frac16 -\frac{1}{\lambda}) \leq g-1 \,. \label{eq: 4}
\end{equation}
Since $V \geq 3$ we have $E=\frac{\lambda V}{2} \geq \frac{3\lambda}{2}$. Applying this to Equation \ref{eq: 4}, we obtain $\lambda \leq 2(2g+1)$. Note that $F\geq 3$ since we assume $\lambda \geq 3$. Then $n\leq 2(2g+1)$ by the exact same argument with the roles of $\lambda$ and $n$ interchanged. In particular, this proves that there are only finitely many possibilities for $\lambda$ and $n$ in the equivar case, for each $g\geq 1$. It remains to show there are only finitely many possibilities for $3 \leq \lambda_1 <\lambda_2$ in the two valency case. Using $\lambda_1 < \lambda$ and $n\geq 4$, Equation \ref{eq: 3} becomes
\begin{equation*}
E(\frac14 - \frac{1}{\lambda_1}) < g-1
\end{equation*}
Since $V_1 \geq 2$ in this case, $E=\lambda_1 V_1 \geq 2\lambda_1$. Applying this we obtain $3\leq \lambda_1 < 2g+2$. Proceeding a similar manner using $n\geq 4$, we have
\begin{equation*}
E(\frac14 - \frac{1}{2\lambda_1} - \frac{1}{2\lambda_2}) \leq g-1
\end{equation*}
and with $\lambda_1\geq 3$ and $E\geq 2\lambda_2$ we obtain $3 < \lambda_2\leq 6g$.

So for $g\geq 1$ there are finitely many possibilities for $\lambda_1, \lambda_2, n$. For $g\geq 2$, all of the data $(V_i, E, F, n, \lambda_i)$ is determined by $g, n, \lambda_1, \lambda_2$ through equations \ref{eq: 1}-\ref{eq: 3}. This proves (i). For $g=1$, $E$ is not determined by $g, n, \lambda_1, \lambda_2$, and there are infinitely many possibilities for $E, V,$ and $F$. This proves (ii).

For $g=0$, equation~\ref{eq: 3} implies $\frac{1}{\lambda}+\frac{1}{n}>\frac{1}{2}$. Using $n\geq 3$ we have $\lambda < 6$, and from $\lambda \geq 3$ we obtain $n<6$. So there are only finitely many possibilities for $n$ and $\lambda$ in the equivar case. In the case of two valencies $3\leq \lambda_1 < \lambda_2$, applying $n\geq 4$ and $\lambda_1 <\lambda$, $\frac{1}{\lambda}+\frac{1}{n}>\frac{1}{2}$ implies $\frac{1}{\lambda_1}+\frac{1}{4}>\frac{1}{2}$ or $3\leq \lambda_1 < 4$. So $\lambda_1=3$. Using this with $n\geq 4$, $\frac{1}{\lambda}+\frac{1}{n}>\frac{1}{2}$ implies $\frac{1}{6}+\frac{1}{2\lambda_2}+\frac{1}{4}>\frac{1}{2}$ or $\lambda_2 < 6$. So for $g=0$ there are finitely many possibilities for $\lambda_1, \lambda_2, n$. This proves (iii) since all of the data $(V_i, E, F, n, \lambda_i)$ is determined by $g, n, \lambda_1, \lambda_2$ through Equations \ref{eq: 1}-\ref{eq: 3}.
\end{proof}

\begin{cor}
\label{cor: finiteness1} A closed cell tesselation on the closed surface of genus $0$ which is edge and face transitive having $n\geq 3$, must have data $(V_i, E, F, n, \lambda_i)$ given by one of the rows of the following:
\begin{table}[htp]
\centering
\caption{Possible Tesselations on the Riemann Surface of Genus 0}
\begin{tabular}{|c|c|c|c|c|}
\hline
 \# Faces & Schl\"afli Symbol & \# Vertices & \# Edges & Solid Type
\\
\hline
 2 & $\{n,2\}$ & $n$ & $n$ & dual Hosahedron
\\
 $k$ & $\{4,2\text{-}k\}$ & $k+2$ & $2k$ & double Hosahedron
\\
4 & $\{3,3\}$ & 4 & 6  & Tetrahedron
\\
4 & $\{6,2\text{-}3\}=D\{3,3\}$ & 10 & 12 & double Tetrahedron
\\
 6 & $\{4,3\}$ & 8 & 12 & Cube
\\
 6 & $\{8,2\text{-}3\}=D\{4,3\}$ & 20 & 24 & double Cube
\\
  8 & $\{3,4\}$ & 6 & 12 & Octahedron
\\
  8 & $\{6,2\text{-}4\}=D\{3,4\}$ & 18 & 24 & double Octahedron
\\
  12 & $\{5,3\}$ & 20 & 30 & Dodecahedron
\\
  12 & $\{10,2\text{-}3\}=D\{5,3\}$ & 50 & 60 & double Dodecahedron
\\
 12 & $\{4,3\text{-}4\}$ & 14 & 24 & Rhombic Dodecahedron
\\
  20 & $\{3,5\}$ & 12 & 30 & Icosahedron
\\
  20 & $\{6,2\text{-}5\}=D\{3,5\}$ & 42 & 60 & double Icosahedron
\\
 30 & $\{4,3\text{-}5\}$ & 32 & 60 & Rhombictriacontahedron
 \\
\hline
\end{tabular}
\label{table:Genus0}
\end{table}

\end{cor}

\noindent {\it Note:} All cases in table~\ref{table:Genus0} are realized in our construction, in particular these cases exist. For example, these are generated by $\Sigma_5$, see table 9 of \cite{HP}.

\begin{proof}
The proof of Theorem~\ref{thm: finiteness} provides an algorithm for finding all of the cases. We have the infinite families $\{n,2\}$ and $\{4,2\text{-}k\}$. Considering $\lambda_1, \lambda_2\geq 3$, in the equivar cases we found the bounds $3\leq n, \lambda \leq 5$, or $n, \lambda = 3,4$, or $5$. Then on each pair $n, \lambda$ we use equation \ref{eq: 3} to solve for $E$, or rule out the case if no integer solution is found. If an integer $E$ is found, we then use equations \ref{eq: 1} and \ref{eq: 2} to solve for $V$ and $F$. This leads to the well-known five Platonic Solids as shown in the chart. We also have the doubles of these five. In the two valency case $3\leq \lambda_1 < \lambda_2$ we found $n=4, \lambda_1=3, \lambda_2=4$ or $5$. Again we use equation \ref{eq: 3} to solve for $E$, then Equations \ref{eq: 1} and \ref{eq: 2} to solve for $V$ and $F$. This provides all cases.
\end{proof}

\begin{cor} If $G$ is an odd order group then no surface of genus $0$ (sphere) occurs in $Y(G)$ and hence $X(G)$ is a $K(\pi,1)$-space. On the other hand, using the odd order theorem, it follows
that for any nonabelian simple group $G$, there exists a surface of genus $0$ (sphere) in the decomposition $Y(G)$ and $\pi_2(X(G)) \neq 0$.
\end{cor}
\begin{proof}
The possible cell-structures of the surface of genus $0$ that can arise in $\hat{Y}(G)$ are captured in Corollary~\ref{cor: finiteness1}. All of these have
an even number of faces or an even number of edges which implies $|G|$ is even if one of these occurs in $\hat{Y}(G)$ by Proposition~\ref{pro: divisibility}.
Thus if $G$ is an odd order group, no spheres occur in $Y(G)$ and hence $X(G)$ is homotopy equivalent to a bouquet of closed surfaces of genus $g \geq 1$ and circles and hence
is a $K(\pi,1)$-space, i.e., all higher homotopy groups vanish. On the other hand, if $G$ is a nonabelian simple group, then by the odd order theorem,
$|G|$ is even and $G$ possesses an element of order $2$. If all elements of order two commuted with each other in $G$, they would form a normal elementary abelian subgroup which is impossible as $G$ is simple and so there exist two noncommuting elements of order two which generate a dihedral subgroup.
As spheres occur in $Y(H)$ when $H$ is dihedral (see the example section under dihedral groups for a proof of this), spheres occur in $Y(G)$ also by monotonicity. Thus $X(G)$ is homotopy equivalent to a bouquet of a positive number of spheres with a $K(\pi,1)$-space and so has $\pi_2(X(G)) \neq 0$.
\end{proof}

Note the odd order theorem was used in the 2nd part of the argument of the last corollary. In fact, if an independent argument could be made to show that
$\pi_2(X(G)) \neq 0$ or equivalently that $Y(G)$ contained a sphere when $G$ is a nonabelian simple group then it would provide a proof of the odd order theorem.

One can similarly solve for all possible Schl\"afli Symbols in higher genus cases. We present the genus $g=1$ case in Corollaries~\ref{cor: finiteness2}. Full lists of the possibilities in the  regular case (which are limited to single valency) are available for genus 2 through 15 and are contained in
work of Conder and Dobcs\'anyi (see \cite{CD}). The reader is warned that in \cite{CD}, cell structures do not have to be closed, that is to say that the interior of faces are open disks but their closure need not be disks in the space:  self-identifications along the boundary are allowed. As noted in the proof of Theorem~\ref{thm: finiteness}, Equation \ref{eq: 3} does not constrain $E$ when $g=1$, and there are infinitely many possibilities for $V, E$, and $F$. For example, it is well-known that a torus can be given a closed-cell tesselation with $pq$ squares, $p, q \geq 2$, (a composite number) simply by subdividing a rectangle into a $p\times q$ grid, then gluing the top and bottom edges, and left and right edges. It is interesting to note that $\Sigma_5$ contains a $g=1$ component with $(V,E,F)=(5,10,5)$ and $\{n,\lambda\}=\{4,4\}$. That is, one can give a closed-cell tesselation on the torus using five 4-gons (a prime number), see Figure~\ref{fig: torus5}. In fact, this is possible with $p$ $4$-gons, for any prime $p\geq 5$. For example Figure~\ref{fig: torus7} shows $p=7$. Note since these are
closed-cell tesselations, the faces are genuinely $4$-gons. (If self-identifications were allowed on the boundary of faces, then such tesselations would
exist trivially.)

\begin{figure}[htp]
\centering
\begin{minipage}{.4\textwidth}
  \centering
  \includegraphics[width=.6\textwidth]{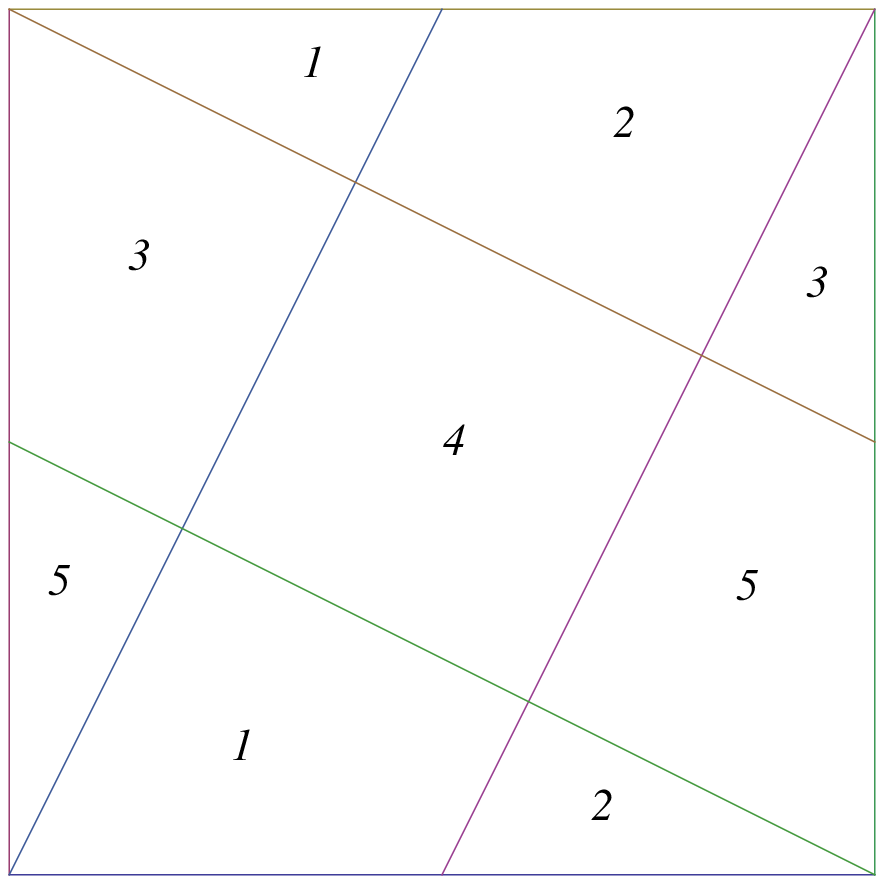}
  \caption{A tesselation of the Torus using five 4-gons.}
  \label{fig: torus5}
\end{minipage}\hspace{1cm}
\begin{minipage}{.4\textwidth}
  \centering
  \includegraphics[width=.6\textwidth]{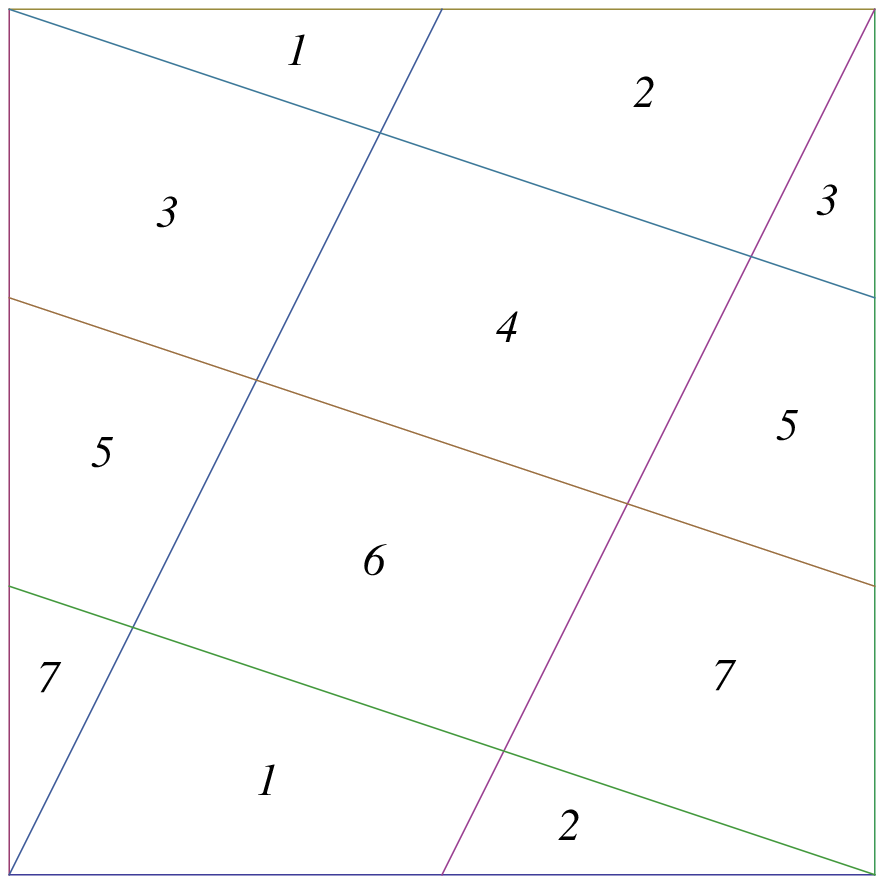}
  \caption{A tesselation of the Torus using seven 4-gons.}
  \label{fig: torus7}
\end{minipage}
\end{figure}

\begin{cor}
\label{cor: finiteness2} A closed-cell tesselation on the closed surface of genus $1$ which is edge and face transitive having $n\geq 3$, must have one of the following Schl\"afli Symbols:
\begin{eqnarray*}
\{3,6\},&& \quad D\{3,6\}=\{6,2\text{-}6\}, \quad \{4,4\}, \quad D\{4,4\}=\{8,2\text{-}4\},
\\
\{6,3\},&& \quad D\{6,3\}=\{12,2\text{-}3\}, \quad  \{4,3\text{-}6\}
\end{eqnarray*}
\end{cor}

\noindent {\it Note:} All cases above exist as tesselations of the torus, however we are not sure if they all occur for $\hat{Y}(G)$, for some group $G$.

\section{Examples}
\label{sec: examples}

\subsection{Dihedral groups}
\label{sec: dihedral}

Let $D_{2n}, n \geq 3$ denote the dihedral group of order $2n$. It consists of $n$ rotations and $n$ reflections which are elements of order $2$. It is generated by two reflections $\sigma_1, \sigma_2$ whose product is a rotation of order $n$. Since a product of an even number of reflections is a rotation and the product of an odd number of reflections is a reflection, each triangle in the triangulation of $X(D_{2n})$ has vertices which consist of two reflections and a rotation. Let $\tau$ denote a generator of the cyclic subgroup of rotations of order $n$ and $\sigma$ a fixed reflection, then $\sigma \tau \sigma=\tau^{-1}$. The set of reflections is then $\{ \sigma \tau^k | 0 \leq k < n \}$.
The center of $D_{2n}$ is trivial if $n$ is odd and has order two generated by $\tau^{\frac{n}{2}}$
if $n$ is even.

Let $\tau^k$ be a noncentral rotation. A simple computation shows that the sheets centered at the vertex $(\tau^k,2)$ have rim vertices of the form \\ $(\sigma \tau^s,1),  (\sigma \tau^{s+k},1), (\sigma \tau^{s+2k},1), \dots$. Thus these sheets form $d$-gons where $d$ is the order of $\tau^k$ and there are $\frac{n}{d}$ of them corresponding to the cosets of the
cyclic group of order $d$ generated by $\tau^k$ in the cyclic group of $n$ rotations.
Each of these sheets fits together with a corresponding sheet of its inverse $(\tau^{-k},2)$
to make a sphere which is the suspension of a $d$-gon. In the cell structure of $\hat{Y}(D_{2n})$, the corresponding Schl\"afli symbol is $\{ d, 2\}$.

Let $\phi(d)$ be Euler's Phi function, denoting the number of primitive $d$th roots of unity or equivalently the number of generators of a cyclic group of order $d$. Then when $n$ is odd,  for every divisor $1 < d | n$ we have $\frac{\phi(d)}{2}$ pairs of elements of order $d$, each pair leading to $\frac{n}{d}$ spheres of Schl\"afli symbol $\{d, 2 \}$ as mentioned above. When $n$ is even we only need to exclude the case $d=2$ which corresponds to the nontrivial central rotation of order $2$. The analysis so far accounts for all $2$-sheets centered at rotations.

Now a completely similar analysis shows that a sheet about $(\tau^k,1)$ is of the form
$(\sigma \tau^s,1), (\sigma \tau^{s+k},2), (\sigma \tau^{s+2k},1), (\sigma \tau^{s+3k},2), \dots$ and hence consists of $d'$ type 1 vertices and $d'$ type 2 vertices where $d'$ is the order of $\tau^{2k}$.The number of such distinct sheets as $s$ varies is $\frac{n}{d'}$ and each of these fits together with a paired sheet about $(\tau^{-k},1)$ to form a sphere which is the suspension of a $2d'$-gon. In the corresponding cell structure in $\hat{Y}(D_{2n})$, the $2$-cells are $4$-gons, the $d'$ equatorial vertices have valency $2$ and the north and south poles $(\tau^{\pm k},1)$ have valency $d'$. Thus the Schl\"afli symbol is $\{ 4, 2\text{-}d' \}$ where $d'$ is the order of $\tau^{2k}$. Now when $\tau^k$ has order $d$ then $\tau^{2k}$ has order $d'=d$ when $d$ is odd and $d'=\frac{d}{2}$ when $d$ is even.

These observations can then be put together to get the following theorem:

\begin{thm}
\label{thm: dihedral}
Let $D_{2n}, n \geq 3$ denote the dihedral group of order $2n$. Then: \\
(1) All components in $Y(D_{2n})$ are spheres (genus $g=0$). \\
(2) If $n$ is odd, then for any $1 <  d |n$ there are $\frac{\phi(d)n}{2d}$ spherical components
with Schl\"afli symbol $\{d,2\}$ and another $\frac{\phi(d)n}{2d}$ spherical components with
Schl\"afli symbol $\{4,2\text{-}d\}=D\{2,d\}$. \\
(3) If $n$ is even, then for any $3 \leq d |n$ there are $\frac{\phi(d)n}{2d}$ spherical components
with Schl\"afli symbol $\{d,2\}$. For every odd such $d$ there is another $\frac{\phi(d)n}{2d}$
spherical components with symbol $\{4,2\text{-}d\}=D\{2,d\}$ and for every even such $d$
there is another $\frac{\phi(d)n}{d}$ spherical components with symbol $\{4,2\text{-}\frac{d}{2}\}$.
\end{thm}

\begin{cor}
$\hat{Y}(D_6)=\hat{Y}(\Sigma_3)$ consists of two spheres with cell structure of type
$\{3,2\}$ and $\{4,2\text{-}3\}=D\{2,3\}$. \\
$\hat{Y}(D_8)$ consists of three spheres with cell structure of type $\{4,2\}$.
\end{cor}

\subsection{Quaternions}

Let $Q_8=\{ \pm 1, \pm i, \pm j , \pm k \}$ denote the quaternionic group of order $8$.
One component of $Y(Q_8)$ is an octahedron with $(k,2), (-k,2)$ as north and south pole and
with the 4 vertices $(i,1), (j,1), (-i,1), (-j,1)$ along the equator. There are two more similar components obtained by cyclically permuting the roles of $i, j$ and $k$. In the corresponding cell structure these three spheres have Schl\"afli symbol $\{4,2\}$.

Thus $\hat{Y}(D_8)$ and $\hat{Y}(Q_8)$ are isomorphic as cell complexes. It is not hard to check that $Y(D_8)$ and $Y(Q_8)$ are isomorphic as simplicial complexes and so are $X(D_8)$ and
$X(Q_8)$ which consist of three octahedra which pairwise meet in a pair of antipodal vertices.
Up to homotopy equivalence we have $X(Q_8) \simeq X(D_8) \simeq \vee_{[3]} S^2 \vee \vee_{[4]} S^1$.

\subsection{Extraspecial $p$-groups}
\label{sec: extraspecial}

Let $p$ be an odd prime and $\mathbb{F}_p$ be the field of $p$ elements. Consider $U_3(p)$ the group of $3 \times 3$ upper triangular matrices with entries in $\mathbb{F}_p$ and $1$'s on the diagonal. Thus
$$
U_3(p)=\{ \begin{bmatrix} 1 & a & b \\ 0 & 1 & c \\ 0 & 0 & 1 \end{bmatrix} |\ a,b,c \in \mathbb{F}_p \}
$$
It is easy to see that $U_3(p)$ has order $p^3$ and exponent $p$ (any nonidentity element has order $p$). To see this just note that any matrix in $U_3(p)$
can be written as $\mathbb{I} + \mathbb{A}$ where $\mathbb{A}$ is a strictly upper triangular $3 \times 3$ matrix and so $\mathbb{A}$ is nilpotent with
$\mathbb{A}^3=0$. Then $(\mathbb{I}+\mathbb{A})^p = \mathbb{I}$ follows from the binomial theorem and the fact that $\binom p 1$ and
$\binom p 2$ are congruent to zero modulo $p$.  $G=U_3(p)$ is sometimes called the extra special $p$-group of exponent $p$ and order $p^3$.

If one sets $x = \begin{bmatrix}  1 & 1 & 0 \\ 0 & 1 & 0 \\ 0 & 0 & 1 \end{bmatrix}$ and $y=\begin{bmatrix}  1 & 0 & 0 \\ 0 & 1 & 1 \\ 0 & 0 & 1 \end{bmatrix}$
then one can readily check that $x$ and $y$ generate $U_3(p)$ and have commutator $[x,y]=xyx^{-1}y^{-1} = \begin{bmatrix} 1 & 0 & 1 \\ 0 & 1 & 0 \\ 0 & 0 & 1 \end{bmatrix}$ which is central and in fact generates the center of $U_3(p)$ which is a cyclic group of order $p$.
From this one can easily verify that $U_3(p)$ has presentation  $<x, y | x^p=y^p=[x,y]^p=[[x,y],x]=[[x,y],y]=1 >$. The abelianization of $U_3(p)$ is an
elementary abelian $p$-group of rank 2 (i.e. vector space of dimension $2$ over $\mathbb{F}_p$) generated by the images $\bar{x}, \bar{y}$ of $x$ and $y$.
Thus $U_3(p)$ fits into a central short exact sequence:

$$
1 \to C \to U_3(p) \to E=\mathbb{F}_p \times \mathbb{F}_p \to 1
$$
where $C$ is the center of $U_3(p)$ and is cyclic of order $p$. It is clear from this short exact sequence that $C=Frat(G)$, the Frattini subgroup of $G$. Corresponding to this extension is a commutator map $[\cdot, \cdot]: E \times E \to C$
which takes two elements of $E$, lifts them to $U_3(p)$ and looks at their commutator which lies in $C$. This map is readily checked to be well-defined
and bilinear, alternating, (see for example \cite{BrP} for details). From this map it follows that any two elements $u, v$ that do not commute in $U_3(p)$ must map
to a basis in $E$ under the abelianization and hence generate $U_3(p)$ by properties of Frattini quotients. They clearly satisfy the same
presentation that $x$ and $y$ did and hence there must be an automorphism of $U_3(p)$ taking any noncommuting pair of elements to any other noncommuting pair of elements. Finally let us note that if $\alpha$ and $\beta$ are conjugate in $U_3(p)$ they must map to the same element in the abelianization $E$.
Thus $\beta=\alpha c$ where $c$ is some element in the center $C$. In particular conjugate elements commute in $U_3(p)$ and every noncentral
element has exactly $p$ elements in its conjugacy class. In fact if $x$ is a non central element, $xC$ is its conjugacy class.

We summarize these observations in the next lemma as we will use them in determining the 2-cell structure of the components of $\hat{Y}(U_3(p))$.

\begin{lem}
\label{lem: extraspecial}
Let $p$ be an odd prime and $G=U_3(p)$ be the extraspecial group of order $p^3$ and exponent $p$. \\
(1) If $(u,v)$, $(u',v')$ are pairs of noncommuting
elements in $G$, then there exists an automorphism $\phi$ of $G$ such that $\phi(u)=u'$ and $\phi(v)=v'$. \\
(2) If $\alpha, \beta$ are conjugate in $G$, then $\alpha$ and $\beta$ commute. \\
(3) Every noncentral element $\alpha$ has exactly $p$ elements in its conjugacy class. \\
(3) If $\alpha$ does not commute with $x$ then the $\alpha$-conjugacy class of $x$ consists of exactly $p$ elements.
\end{lem}
\begin{proof}
Parts (1), (2) and (3) were proved in the paragraph before the lemma. Part (4) follows as the size of an $\alpha$-conjugacy class
must divide the order of $\alpha$. If $\alpha$ does not commute with $x$ then $\alpha$ is not the identity element and hence has order $p$
and the $\alpha$-conjugacy class of $x$ must have size $> 1$ and dividing the order of $\alpha$. As $|\alpha|=p$ is prime, this size must be $p$.
\end{proof}

The next theorem determines the structure of $\hat{Y}(U_3(p))$ completely.

\begin{thm}
\label{thm: extraspecial}
Let $p$ be an odd prime and $G=U_3(p)$, the extra special $p$-group of order $p^3$ and exponent $p$.  Then in
$\hat{Y}(G)$ we have: \\
(1) All components are isomorphic as cell-complexes and there are $\frac{(p^2-1)(p^2-p)}{2}$ of them. \\
(2) The cell structure of each component of $\hat{Y}(G)$ is a regular abstract $3$-polytope which Schl\"afli symbol $\{2p,p\}$. \\
(3) These cell structures hence tessellate the Riemann surface of genus $g=\frac{p(p-3)}{2}+1$ with $2p$-gons. The  face, edge and vertex count of this regular tesselation is given by
 $F=p, E=p^2, V=2p$ and vertex valency $p$.
\end{thm}
\begin{proof}
If $T$ is the component of $Y(G)$ determined by the triangle \\ $[(x,1), (y,1), (xy,2)]$ and $T'$ is the component of $Y(G)$ determined by the
triangle $[(u,1), (v,1), (uv,2)]$, then Lemma~\ref{lem: extraspecial} guarantees the existence of a group automorphism and hence simplicial automorphism
of $Y(G)$ (and also $X(G)$) which takes one triangle to the other and hence induces a simplicial isomorphism of the component $T$ with the component $T'$.
This simplicial isomorphism induces a cell-isomorphism between the cell structures of these two components in $\hat{Y}(G)$ also.

In the component determined by the triangle $[(x,1),(y,1),(xy,2)]$, all type $2$-vertices are conjugate to $xy$. Since the  $x$-conjugacy class of $xy$
must all occur as type 2-vertices by Proposition~\ref{pro: star1} there are at least $p$ of these in the component. However $xy$ only has $p$-conjugates so all conjugates of $xy$ occur. As $x$ and $y$ do not commute, they are not conjugate by Lemma~\ref{lem: extraspecial}. Thus by Proposition~\ref{pro: star2},
the sheets in the closed star of $(xy,2)$ in that component consist of two distinct $xy$-conjugacy classes of triangles each of order $p$ and hence forms a
$2p$-gon in the corresponding cell-structure. Now there is a $2p$-gon face in the component for each type-2 vertex in that component and these consist
of the $p$ conjugates of $(xy,2)$ with possible multiplicities due to the desingularization process (there can be more than one type-2 vertex labeled with the same group element in a given component because of the desingularization).

We wish to show there is no multiplicity of type $2$-vertices in a fixed component and hence that there are exactly $p$ of these faces.
To do this, by symmetry (we have already shown there are automorphisms which will take any triangle to any other in $X(G)$), it is enough to show that
there is exactly one sheet about a type-2 vertex labelled $(xy,2)$ in the pseudo-manifold component of $X(G)$ that contains the triangle $[(x,1),(y,1),(xy,2)]$ (before desingularization). In other words there is only one disk in the closed star of $(xy,2)$ which is a bouquet of disks that lies in that pseudomanifold component. Note that all the type-1 vertices in the given  pseudomanifold-component are conjugate
to either $x$ or $y$ by Theorem~\ref{thm: conjugate1} and hence of the form $xz^a, yz^b$ where $z$ is a generator of the center and $0 \leq a, b < p$. Thus any triangle in a sheet about a
type-2 vertex labeled $(xy,2)$ in the given pseudo-manifold component of $X(G)$ (before desingularization) has to be of the form $[(xz^a,1), (yz^b,1), (xy,2)]$ or of the form $[(yz^b, 1), (y^{-1}xyz^a, 1), (xy,2)]$ which forces $b \equiv -a \text{ mod } p$. This means there there are at most $2p$ such triangles and hence exactly one such sheet in that pseudo manifold component.
Thus there are no multiply labelled type $2$-vertices in a given component in $Y(G)$ and we can conclude that each component of $\hat{Y}(G)$ has
cell-structure given by exactly $p$ many $2p$-gon faces.

By part (5) of Theorem~\ref{thm: polytope} and Lemma~\ref{lem: extraspecial} we see that the cell structure of each component forms a {\bf regular}
abstract $3$-polytope. By Theorem~\ref{thm: conjugate2} the vertex valency is the size of the $x$-conjugacy class of $xy$ which is $p$.
Thus this regular abstract $3$-polytope has Schl\"afli symbol $\{ 2p, p \}$. As $F=p$, the formulas in part (3) of the same theorem can then be used to yield
the stated values of $E, V$ and $g$.

As $V=2p$ in the cell structure, we see that the $p$ conjugates of $(x,1)$ and $(y,1)$ occur without multiplicity in the component (alternatively one can mimic the proof used for type 2-vertices earlier). Thus each component consists exactly of two noncommuting conjugacy classes in $U_3(p)$ and is determined by the unordered pair of these. These pairs in turn are determined uniquely by the two linearly independent vectors in the abelianization $E$ of $U_3(p)$ that they project to. Thus the number of components of $Y(U_3(p))$ is the same as the number of unordered pairs of noncommuting conjugacy classes which is the same as the number of unordered basis of $E$, a $\mathbb{F}_p$-vector space of dimension 2. This number is easily computed as
$\frac{(p^2-1)(p^2-p)}{2}$.

\end{proof}

Note when $p=3$ in the last theorem, we see that $U_3(3)$ provides a regular tesselation of the
torus by $3$ hexagons with vertex valancy $3$. When $p=5$, $U_3(5)$ provides a regular
tesselation of the surface of genus $6$ by five $10$-gons with vertex valancy $5$.

\begin{cor}
As one varies the construction $X(G)$ over all finite nonabelian groups, the set of genuses of components that occur is infinite.
\end{cor}
\begin{proof}
Theorem~\ref{thm: extraspecial} shows that the set of genuses obtained when looking at extraspecial groups over all odd primes $p$ is infinite and so the corollary follows.
\end{proof}

For $G$ a group, let $\widetilde{Aut}(G)$ denote the extended automorphism group which consists of automorphisms and anti-automorphims of the group $G$ under composition.
$Aut(G)$ is a subgroup of $\widetilde{Aut}(G)$ of index $2$ (hence normal) as any anti-automorphism of $G$ is the composition of the inversion map $I(x)=x^{-1}$ with an automorphism. Furthermore it is easy to verify that if $\phi$ is an automorphism then
$\phi(x^{-1})^{-1}=\phi(x)$ and so $\phi$ commutes with $I$. Thus $I$ generates a central
subgroup of order $2$ complementary to $Aut(G)$ in $\widetilde{Aut}(G)$ and so
$\widetilde{Aut}(G) \cong Aut(G) \times \mathbb{Z}/2\mathbb{Z}$.

The anti-automorphism group of $U_3(p)$ acts transitively on the set of components of
$\hat{Y}(G)$. If $H$ is the stabilizer of a component in this action, then $H$ acts transitively
on the flags in the cell-structure of this component by Theorem~\ref{thm: polytope} and its proof. As the group elements represented by the vertices of this component generate the group $U_3(p)$, this action is also faithful. Thus $H$ is the abstract flag symmetry group of the
regular abstract $3$-polytope represented by this component. Just as the isometry group of classical regular polyhedra are reflection groups, it has been shown that the automorphism groups of abstract regular polyhedra are generated by involutions. It follows from this theory (see \cite{MS}) that $H$ is a C-string group generated by three involutions $\psi_0, \psi_1,\psi_2$ such that $(\psi_0\psi_1)^{2p}=1$ and $(\psi_1\psi_2)^p=1$ and $\psi_0, \psi_2$
commute.

Thus we have obtained the following corollary also:

\begin{cor}
Let $p$ be an odd prime and $G=U_3(p)$ be the extraspecial group of exponent $p$ and order $p^3$. Then there exists a subgroup $H$ of $Aut(G) \times \mathbb{Z}/2\mathbb{Z}$
of index $\frac{(p^2-1)(p^2-p)}{2}$ such that $H$ is a C-string group generated by three involutions
$\psi_0, \psi_1, \psi_2$ which satisfy $(\psi_0\psi_1)^{2p}=1=(\psi_1\psi_2)^p$ and
$\psi_0, \psi_2$ commute.
\end{cor}

\bigskip

\noindent
Dept. of Mathematics \\
University of Rochester, \\
Rochester, NY 14627 U.S.A. \\
E-mail address: jonpak@math.rochester.edu \\

\bigskip

\noindent
Dept. of Mathematics \\
University of Rochester, \\
Rochester, NY 14627 U.S.A. \\
E-mail address: herman@math.rochester.edu \\

\bigskip

\noindent
Dept. of Mathematics \\
Bilkent University, \\
Ankara, Turkey. \\
E-mail address: yalcine@fen.bilkent.edu.tr \\

\end{document}